\documentclass[12pt]{amsart}

\newtheorem{theorem}{Theorem}[section]

\newtheorem{lemma}[theorem]{Lemma}

\newtheorem{proposition}[theorem]{Proposition}
\newtheorem{corollary}[theorem]{Corollary}

\theoremstyle{remark}

\newtheorem*{remark}{Remark}

\DeclareMathOperator{\cont}{cont} 
\DeclareMathOperator{\den}{den} 
\DeclareMathOperator{\num}{num} 
\DeclareMathOperator{\ord}{ord} 

\newcommand{\bC}{\mathbb{C}}
\newcommand{\bG}{\mathbb{G}}

\newcommand{\bN}{\mathbb{N}}

\newcommand{\bQ}{\mathbb{Q}}
\newcommand{\bR}{\mathbb{R}}
\newcommand{\bZ}{\mathbb{Z}}

\newcommand{\cP}{{\mathcal{P}}}
\newcommand{\cO}{{\mathcal{O}}}

\newcommand{\ga}{\mathfrak{a}}
\newcommand{\gb}{\mathfrak{b}}
\newcommand{\gp}{\mathfrak{p}}

\newcommand{\tP}{\tilde{P}}
\newcommand{\trho}{\tilde{\rho}}

\newcommand{\tS}{\tilde{S}}

\newcommand{\ue}{\mathbf{e}}

\newcommand{\ui}{\mathbf{i}}

\newcommand{\uxi}{\underline{\xi}}

\newcommand{\Cmult}{\bC^\times}
\newcommand{\Ctor}{\Cmult_\mathrm{tor}}
\newcommand{\CT}{\bC[T]}
\newcommand{\disp}{\displaystyle}
\newcommand{\et}{\quad\mbox{and}\quad}
\newcommand{\gen}[1]{\langle #1\rangle}
\newcommand{\Ga}{\bG_\mathrm{a}}
\newcommand{\Gm}{\bG_\mathrm{m}}
\newcommand{\Gtor}{\bG_\mathrm{tor}}
\newcommand{\KT}{K[T]}

\newcommand{\ndiv}{\not\hskip-1pt|\hskip3pt}
\newcommand{\Qbar}{\overline{\bQ}}

\newcommand{\ZT}{\bZ[T]}

\begin{document}

\baselineskip=16.8pt 

\title[Small value estimates]
{Small value estimates for the multiplicative group}
\author{Damien ROY}
\address{
   D\'epartement de Math\'ematiques\\
   Universit\'e d'Ottawa\\
   585 King Edward\\
   Ottawa, Ontario K1N 6N5, Canada}
\email{droy@uottawa.ca}
%
%
\thanks{Work partially supported by NSERC and CICMA}

\begin{abstract}
We generalize Gel'fond's transcendence criterion to the context of a
sequence of polynomials whose first derivatives take small values on
large subsets of a fixed subgroup of the multiplicative group
$\Cmult$ of $\bC$.
\end{abstract}

\maketitle

%
%

\section{Introduction}
\label{sec:intro}

For applications to transcendental number theory, it would be
desirable to extend the actual criteria for algebraic independence
so that they deal more efficiently with polynomials taking small
values on large subsets of a finitely generated subgroup of an
algebraic group.  At the moment, one could say that these criteria
concentrate on the smallest non-zero value of each polynomial on
such sets, regardless of the global distribution of values. A good
illustration of the need for refined criteria, and our main
motivation for this quest, is a conjectural small value estimate for
the algebraic group $\Ga\times\Gm$ which is proposed in \cite{R2001}
and shown to be equivalent to Schanuel's conjecture. In a preceding
paper \cite{ixi}, we explored the case of the additive group $\Ga$.
Here, we turn to the multiplicative group $\Gm$. Although this is
again an algebraic group of dimension one, we will see that it
presents new challenges as roots of unity come into play.

Let $\Cmult$ denote the multiplicative group of non-zero complex
numbers, let $m$ be a positive integer, and let $\xi_1, \dots, \xi_m
\in \Cmult$.  An application of Dirichlet's box principle shows
that, for any non-negative real numbers $\beta$, $\sigma$, $\tau$
and $\nu$ with
\begin{equation}
 \label{intro:conditions_Dirichlet}
 m\sigma+\tau < 1,
 \quad
 \beta> (m+1)\sigma+\tau
 \et
 \nu< 1+\beta-m\sigma-\tau,
\end{equation}
and for any positive integer $n$ which is sufficiently large in
terms of the preceding data, there exists a non-zero polynomial
$P\in\ZT$ of degree at most $n$ and height at most $\exp(n^\beta)$
satisfying $|P^{[j]}(\xi_1^{i_1}\cdots\xi_m^{i_m})| < \exp(-n^\nu)$
for each choice of integers $i_1,\dots,i_m$ and $j$ with $0\le
i_1,\dots,i_m\le n^\sigma$ and $0\le j < n^\tau$.  Here the
\emph{height} of $P$, denoted $H(P)$, is defined as the maximum of
the absolute value of its coefficients divided by their greatest
common divisor, and the expression $P^{[j]}$ stands for the $j$-th
divided derivative of $P$ (see \S \ref{sec:prelim}). The goal of
this paper is to establish the following partial converse to this
statement.

\begin{theorem}
 \label{intro:thm1}
Let $m$ be a positive integer, let $\xi_1,\dots,\xi_m$ be non-zero
multiplicatively independent complex numbers which generate over
$\bQ$ a field of transcendence degree one, and let $\beta, \sigma,
\tau, \nu \in \bR$ with
\begin{gather}
 \label{intro:thm1:eq1}
 \sigma \ge 0,
 \quad
 \tau \ge 0,
 \quad
 \frac{5m+1}{m+5}\,\sigma + \tau < 1,
 \quad
 \beta \ge 1 + \sigma, \\[5pt]
 \label{intro:thm1:nu}
 \nu >
   \left\{
   \begin{array}{ll}
   \disp
   1 + \beta - \frac{3m-1}{m+5}\,\sigma - \tau
     &\text{if $m\ge 2$,} \\[12pt]
   \disp
   1 + \beta - \frac{5}{11}\sigma -\tau
     &\text{if $m=1$.}
   \end{array}
   \right.
\end{gather}
Then, for infinitely many positive integers $n$, there exists no
non-zero polynomial $P\in \ZT$ with $\deg(P)\le n$ and $H(P) \le
\exp(n^\beta)$ such that
\begin{equation}
 \label{intro:thm1:eq2}
 \max\left\{ |P^{[j]}(\xi_1^{i_1}\cdots\xi_m^{i_m})| \,;\, 0\le
 i_1,\dots,i_m\le n^\sigma,\ 0\le j < n^\tau \right\} < \exp(-n^\nu).
\end{equation}
\end{theorem}

When $m=1$ and $\sigma=\tau=0$, the above result reduces to the
well-known Gel'fond's transcendence criterion. So, for $m=1$, it
provides a gain of $(5/11)\sigma+\tau$ in the estimate for $\nu$
compared to Gel'fond's criterion. For $m\ge 2$, the gain is
$((3m-1)/(m+5))\sigma+\tau$.  On the other hand, the conditions
\eqref{intro:conditions_Dirichlet} of application of Dirichlet's box
principle put an upper bound on the gain that can be achieved. It
suggests the possibility that Theorem \ref{intro:thm1} remains true
for any integer $m\ge 1$ with the condition on $\nu$ relaxed to
$\nu> 1+\beta-m\sigma-\tau$, when $m\sigma+\tau<1$, but we have not
been able to prove this.  Note that, when $\sigma=0$, Theorem
\ref{intro:thm1} deals with finitely many points and then it follows
from Proposition 1 of \cite{LR}. The novelty here is that we deal
with large numbers of points.

The proof of the above result is involved but the main underlying
idea is simple and is inspired by techniques from zero estimates. If
a polynomial $P\in\ZT$ takes small values at all points of the form
$\xi^a$ with $\xi$ in a subset $E$ of $\Cmult$ and $a$ in a subset
$A$ of $\bN^*$, then the polynomials $P(T^a)$ with $a\in A$ take
small values at all points of $E$. Applying Corollary 3.2 of
\cite{ixi}, one deduces that the product $\prod_{\xi\in E} |Q(\xi)|$
is small, where $Q(T)$ denotes the greatest common divisor in $\ZT$
of the polynomials $P(T^a)$ with $a\in A$.  However, for this to be
useful, we also need good upper bounds for the degree and height of
$Q(T)$. The precise result that we use for this purpose is stated
and proved in Section \S\ref{sec:gcd}.  For simplicity, we just
mention here the following consequence of it, where $\Ctor$ stands
for the group of roots of unity, the torsion part of $\Cmult$.

\begin{theorem}
 \label{intro:thm_gcd}
Let $\beta,\delta,\mu \in \bR$ with $0<\delta$, $0<\mu<1$ and
$1+\mu<\beta$.  Let $n$ be a positive integer, let $A$ be the set of
all prime numbers $p$ with $p \le n^\mu$, let $P$ be a non-zero
polynomial of $\ZT$ of degree at most $n$ and height at most
$\exp(n^\beta)$ with no root in $\Ctor\cup\{0\}$, and let $Q\in\ZT$
be a greatest common divisor of the polynomials $P(T^a)$ with $a\in
A$. If $n$ is sufficiently large as a function of $\beta$, $\delta$
and $\mu$, we have $\deg(Q) \le n^{1-\mu+\delta}$ and $H(Q) \le
\exp(n^{\beta-2\mu+\delta})$.
\end{theorem}

This result is the multiplicative analog of Theorem 1.2 of
\cite{ixi}.  To achieve such non-trivial estimates on the degree and
height of $Q$, the requirement that $P$ has no root in
$\Ctor\cup\{0\}$ is necessary. For example, if $P(T)$ is of the form
$T^r(T^s-1)$ for some integers $r\ge 0$ and $s\ge 1$, then $P(T)$
divides $P(T^a)$ for any integer $a\ge 1$, and so $P(T)$ itself is
the gcd of the latter collection of polynomials.

In practice, we start with a polynomial $P$ satisfying
\eqref{intro:thm1:eq2} and we take for $E$ a suitable subset of the
subgroup of $\Cmult$ generated by $\xi_1,\dots,\xi_m$.  In order to
get appropriate degree and height estimates for the corresponding
polynomial $Q$, we first need to remove from $P$ a suitable
cyclotomic factor. General estimates for this are given in
\S\ref{sec:first}. They require a lower bound for the absolute value
of the cyclotomic factor on the set $E$. This is easy to achieve if
one assumes that $\xi_1,\dots,\xi_m$ do not all have absolute value
one, but the general case requires more elaborate arguments which
occupy all of \S\ref{sec:cyclo} and \S\ref{sec:reduction} for the
case $m\ge 2$, and most of \S\ref{sec:rank_one} in the case $m=1$.
The proof of Theorem \ref{intro:thm1} is completed in
\S\ref{sec:mpoints} for $m\ge 2$ and in \S\ref{sec:product} for
$m=1$.  In both case, we end up with a product $\prod_{\xi\in E}
|Q(\xi)|$ being small and we need to choose $\xi\in E$ such that
$|Q(\xi)|$ is small in order to be able to apply a standard
transcendence criterion. The refined estimate that we obtain in the
case $m=1$ follows by observing that these values $|Q(\xi)|$ cannot
be uniformly small.  For this we use a combinatorial result proved
in \S\ref{sec:Zaran} as an extension of Proposition 9.1 of
\cite{ixi}.

%
%
\section{Notation and preliminaries}
\label{sec:prelim}

Throughout this paper, the symbols $i,j,k$ are restricted to
integers.  We denote by $\Cmult$ the multiplicative group of
non-zero complex numbers, by $\Ctor$ its torsion subgroup, by $\bN$
the set of non-negative integers, and by $\bN^*$ the set of positive
integers. We also denote by $|E|$ the cardinality of an arbitrary
set $E$, and by $\phi$ the Euler totient function.  A
\emph{cyclotomic} polynomial is a monic polynomial of $\ZT$ whose
roots lie in $\Ctor$.  For any integer $j\ge 0$, we define the
$j$-th \emph{divided} derivative of a polynomial $P\in\CT$ by
$P^{[j]}=(j!)^{-1}P^{(j)}$ where $P^{(j)}=d^jP/dT^j$ is the usual
$j$-th derivative of $P$. Finally, the \emph{length} $L(P)$ of a
polynomial $P\in \bC[T_1,\dots,T_m]$ is the sum of the absolute
values of its coefficients.

Let $K$ be a number field and let $d=[K:\bQ]$.  For each place $v$
of $K$, we normalize the corresponding $v$-adic absolute value $|\
|_v$ of $K$ so that it extends the usual absolute value of $\bQ$ if
$v$ is Archimedean, or the usual $p$-adic absolute value of $\bQ$
with $|p|_v=p^{-1}$ if $v$ lies above a prime number $p$.  We also
denote by $K_v$ the completion of $K$ at $v$, and by $d_v$ its local
degree.  For any polynomial $P\in K_v[T_1,\dots,T_m]$, we define the
$v$-adic norm $\|P\|_v$ of $P$ as the largest $v$-adic absolute
value of its coefficients.  Finally we define the \emph{height}
$H(P)$ of any polynomial $P\in K[T_1,\dots,T_m]$ by
\[
 H(P) = \prod_v \|P\|_v^{d_v/d}
\]
where the product extends over all places $v$ of $K$.  This height
is said to be \emph{homogeneous} because it satisfies $H(aP)=H(P)$
for any non-zero element $a$ of $K$, and \emph{absolute} as it is
independent of the choice of the number field $K$ containing the
coefficients of $P$.  It therefore extends to a height on
$\Qbar[T_1,\dots,T_m]$ where $\Qbar$ stands for the algebraic
closure of $\bQ$. In particular, the height of a non-zero polynomial
$P\in \bZ[T_1,\dots,T_m]$ is simply given by $H(P) = \|P\|/\cont(P)$
where $\|P\|=\|P\|_\infty$ is the maximum of the absolute values of
its coefficients (we also use the latter notation for polynomials
with complex coefficients), and where the \emph{content} $\cont(P)$
of $P$ is the gcd of its coefficients.  We say that a non-zero
polynomial of $\bZ[T_1,\dots,T_m]$ is \emph{primitive} if its
content is $1$, and that it is \emph{primary} if it is a power of an
irreducible element of $\bZ[T_1,\dots,T_m]$.  This implies that a
non-constant primary polynomial of $\bZ[T_1,\dots,T_m]$ is
primitive.

In the sequel, we will frequently use the well-known fact that for
one-variable polynomials $P_1,\dots,P_s\in\Qbar[T]$ with product
$P=P_1\cdots P_s$, we have
\begin{equation}
 \label{prelim:ineq_Gelfond}
 e^{-\deg(P)}H(P) \le H(P_1) \cdots H(P_s) \le e^{\deg(P)} H(P).
\end{equation}

For a single point $x\in\Qbar$, we use the same notation $H(x)$ to
denote the \emph{inhomogeneous height} of $x$, that is the height of
the polynomial $T-x$.  For $x\in K$, it is given by the formula
$H(x) = \prod \max\{1,|x|_v\}^{d_v/d}$ where the product runs
through all places $v$ of $K$.  As the field $K$ can be chosen to be
arbitrarily large, this shows that we have $H(x^m) = H(x)^{|m|}$ for
any $m\in\bZ$ and any non-zero $x\in\Qbar$.  From
\eqref{prelim:ineq_Gelfond}, we deduce that, if $x_1, \dots, x_s \in
\Qbar$ are all the roots of a non-zero polynomial $P\in\Qbar[T]$ of
degree $s$, listed with their multiplicities, we have
\begin{equation}
 \label{prelim:ineq_roots}
 e^{-s}H(P) \le H(x_1) \cdots H(x_s) \le e^{s} H(P).
\end{equation}

The following lemma formalizes the standard procedure of
``linearization'' while handling multiplicities at the same time
(cf. \cite[Lemma 2.1]{ixi}).

\begin{lemma}
 \label{lemma:linearization}
Let $\varphi \colon \ZT \to [0,\infty)$ be a multiplicative
function, let $\delta$, $d$ and $Y$ be positive real numbers with
$\delta<1$ and $e^d \le Y$, and let $t \in \bN^*$. Suppose that
there exists a non-zero polynomial $Q_1\in\ZT$ of degree at most $d$
and height at most $Y$ for which $Q = \gcd\{Q_1^{[j]}(T) \,;\, 0\le
j <t\}$ satisfies $\varphi(Q) \le \delta$. Then, there exists a
primary polynomial $S\in\ZT$ with
\begin{equation*}
 \label{lemma:lin:eq2}
 \deg(S) \le d/t,
 \quad
 H(S) \le Y^{2/t}
 \et
 \varphi(S) \le \delta^{1/(6t)}.
\end{equation*}
\end{lemma}

By \emph{multiplicative}, we mean that the function $\varphi$
satisfies $\varphi(FG)=\varphi(F)\varphi(G)$ for any $F,G\in\ZT$. In
our applications later, $\varphi$ takes the form
$\varphi(P)=\prod_{\xi\in E} |P(\xi)|$ for some fixed finite set of
complex numbers $E$.

\begin{proof}
Let $Q=R_1\cdots R_s$ be a factorization of $Q$ into irreducible
elements of $\ZT$.  Since $Q$ divides $Q_1$, we find
\[
 \prod_{i=1}^s \Big( Y^{\deg(R_i)} H(R_i)^d \Big)
 \le Y^{\deg(Q_1)} \Big(e^{\deg(Q_1)} H(Q_1) \Big)^d
 \le Y^{3d}.
\]
Therefore, upon writing $\delta = Y^{-3d\eta}$ for an appropriate
value of $\eta > 0$, we obtain
\[
 \prod_{i=1}^s \varphi(R_i)
 = \varphi(Q)
 \le Y^{-3d\eta}
 \le \prod_{i=1}^s \Big( Y^{\deg(R_i)} H(R_i)^d \Big)^{-\eta}.
\]
So, there is at least one index $i$ with $1\le i\le s$ such that the
polynomial $R=R_i$ satisfies
\begin{equation}
 \label{lemma:lin:eq3}
 \varphi(R) \le \Big( Y^{\deg(R)} H(R)^d \Big)^{-\eta}.
\end{equation}
Since $R$ divides $Q_1^{[j]}$ for $j=0,\dots,t-1$, the polynomial
$Q_1$ is divisible by $R^t$.  This implies that $\deg(R) \le d/t$
and $H(R)^t \le e^d H(Q_1) \le Y^{2}$.  Let $k\ge 1$ be the largest
integer such that the polynomial $S=R^k$ satisfies $\deg(S)\le d/t$
and $H(S)\le Y^{2/t}$ (such an integer exists since $R\neq \pm 1$).
We consider two cases. If $\deg(S) \ge d/(2t)$, then
\eqref{lemma:lin:eq3} leads to $\varphi(S) \le Y^{-\eta\deg(S)} \le
Y^{-\eta d/(2t)} = \delta^{1/(6t)}$.  On the other hand, if $\deg(S)
< d/(2t)$, we have $\deg(R^{2k})\le d/t$ and so $H(R^{2k}) \ge
Y^{2/t}$.  As $H(R^{2k}) \le e^{\deg(R^{2k})} H(R)^{2k} \le Y^{1/t}
H(R)^{2k}$, we deduce that $H(R)^k \ge Y^{1/(2t)}$ and then
\eqref{lemma:lin:eq3} leads to $\varphi(S) \le H(R)^{-\eta k d} \le
Y^{-\eta d/(2t)} = \delta^{1/(6t)}$, as in the previous case.
\end{proof}

For any finite subset $E$ of $\bC$ with at least two points, we
define
\begin{equation}
 \label{intro:def:deltaE}
 \Delta_E = \prod_{\xi'\neq \xi} |\xi'-\xi|^{1/2}
\end{equation}
where the product is taken over all ordered pairs $(\xi,\xi')$ of
distinct elements of $E$.  When $E$ consists of one point, we put
$\Delta_E=1$.  The following result is a reformulation of Corollary
3.2 of \cite{ixi} and our main tool to study families of polynomials
taking small values on such a set $E$.

\begin{proposition}
 \label{prelim:prop:resultant}
Let $E$ be a non-empty finite set of complex numbers, let $n,\, t
\in \bN^*$ with $n\ge t |E|$, let $P_1,\dots,P_r\in\ZT$ be a finite
sequence of $r\ge 2$ non-zero polynomials of degree at most $n$, and
let $Q\in\ZT$ be their greatest common divisor.  Then we have
\begin{equation}
 \label{corPP:eq1}
 \prod_{\xi\in E} \left( \frac{|Q(\xi)|}{\cont(Q)} \right)^t
 \le
   c_1 \Big( \max_{1\le i\le r} H(P_i) \Big)^{2n}
       \prod_{\xi\in E} \Bigg(
       \max_{\substack{1\le i\le r \\ 0\le j<t}}
       |P_i^{[j]}(\xi)|
       \Bigg)^t,
\end{equation}
with $c_1 = e^{10n^2} (2+c_E)^{4nt|E|} \Delta_E^{-t^2}$, where $c_E
= \max_{\xi\in E} |\xi|$ and $\Delta_E$ is defined by
\eqref{intro:def:deltaE}.
\end{proposition}

We conclude this section by stating the version of Gel'fond's
criterion on which all our results ultimately rely.  It is mainly
due to Brownawell \cite{Br} and Waldschmidt \cite{Wa1} (see the
comments after Lemma 2.2 of \cite{ixi} for more details).

\begin{lemma}
 \label{gelfond:curve}
Let $\alpha$, $\beta$ and $\epsilon$ be positive real numbers with
$\beta\ge \alpha$, and let $\xi_1,\dots,\xi_m$ be a finite sequence
of complex numbers which generate a field of transcendence degree
one over $\bQ$. For infinitely many integers $n$, there exists no
polynomial $P\in \bZ[T_1,\dots,T_m]$ of degree at most $n^\alpha$
and height at most $\exp(n^\beta)$ satisfying
\[
 0 < |P(\xi_1,\dots,\xi_m)|
   \le \exp( -n^{\alpha+\beta+\epsilon} ).
\]
\end{lemma}

%
%
\section{The first step}
\label{sec:first}

The goal of this section is to establish the following result which
represents the first step in the proof of our main theorem.

\begin{proposition}
 \label{first:prop}
Let $M,n,t\in\bN^*$ and $X\in\bR$ with $1\le t\le n$. Let $A$ be a
non-empty subset of $\{1,2,\dots,M\}$, and let $E$ be a non-empty
finite subset of $\Cmult$ with $E \cap \Ctor = \emptyset$.  Finally,
let $P\in\ZT$ be a non-zero polynomial with $\deg(P) \le n$ and
$H(P)\le X$, written as a product $P(T) = P_0(T) T^r \Phi(T)^t$
where $P_0 \in \ZT$, $r\in \bN$ and $\Phi\in\ZT$, with $\Phi$
cyclotomic. Put
\[
 \begin{aligned}
 c_E &= \max \{ \max(|\xi|,|\xi|^{-1})\,;\, \xi\in E\,\},\\
 \delta_\Phi &= \min\{ |\Phi(\xi^a)|\,;\,
   a\in A, \, \xi\in E\,\},\\
 \delta_P &= \max\{ |P^{[j]}(\xi^a)|\,;\,
   a\in A, \, \xi\in E,\, 0\le j <2t-1\,\},
 \end{aligned}
\]
and assume that
\begin{equation}
 \label{first:prop:eq1}
 t|E| \le Mn \le \frac{1}{10} \log X
 \et
 (2+c_E)^{20t|E|} \le X.
\end{equation}
Then the polynomial $Q(T) = \gcd\{P_0^{[j]}(T^a) \,;\, a\in A,\,
0\le j<t\,\}$ (computed in $\ZT$) satisfies
\[
 \prod_{\xi\in E} \frac{|Q(\xi)|}{\cont(Q)}
 \le X^{5Mn/t} \Delta_E^{-t}
    \left( \frac{\delta_P}{\min(1,\delta_\Phi)^{3t}} \right)^{|E|}.
\]
\end{proposition}

In practice, given $P$, we choose $r$ to be the largest non-negative
integer such that $T^r$ divides $P(T)$, and $\Phi(T)$ to be the
cyclotomic polynomial of $\ZT$ of largest degree such that
$\Phi(T)^t$ divides $P(T)$.  Then, we have $Q(0)\neq 0$ and no root
of $Q$ is a root of unity.  As we saw in \S\ref{sec:intro}, such
conditions are required in order to get good estimates on the degree
and height of $Q$.

To prove the above result, we will apply Proposition
\ref{prelim:prop:resultant} to the family of polynomials
$P_0^{[j]}(T^a)$ with $a\in A$ and $0\le j< t$.  In order to
estimate the absolute value of their derivatives at the elements of
$E$, we first establish three lemmas.

\begin{lemma}
 \label{first:lemma:F^(-r)}
Let $\Phi\in\CT$, $t\in\bN^*$ and $\xi\in\bC$ with $\Phi(\xi)\neq
0$. For any integer $j\ge 0$, we have
\[
 \left| \left( \Phi^{-t} \right)^{[j]}(\xi) \right|
 \le
 \frac{1}{j!} \Big( (t+2j) \deg(\Phi) \|\Phi\|
 \max(1,|\xi|)^{\deg(\Phi)}\Big)^j |\Phi(\xi)|^{-t-j}.
\]
\end{lemma}

\begin{proof}
For each $j\ge 0$, the $j$-th derivative of $\Phi^{-t}$ can be
written in the form $\left( \Phi^{-t} \right)^{(j)} = A_j
\Phi^{-t-j}$ where $A_j$ is a polynomial of $\CT$ satisfying $A_0=1$
for $j=0$, and the recurrence relation $A_j = A'_{j-1}\Phi
-(t+j-1)A_{j-1}\Phi'$ for $j\ge 1$. If $j\ge 1$, this gives
$\deg(A_j) \le \deg(A_{j-1})+\deg(\Phi)$ and by recurrence we get
$\deg(A_j)\le j\deg(\Phi)$ for each $j\ge 0$. For the length of
these polynomials, we also find, for $j\ge 1$,
\[
 \begin{aligned}
 L(A_j)
 &\le L(A'_{j-1})\|\Phi\| + (t+j-1)L(A_{j-1})\|\Phi'\| \\
 &\le \big(\deg(A_{j-1})+(t+j-1)\deg(\Phi)\big) \|\Phi\| L(A_{j-1})\\
 &\le (t+2j-2) \deg(\Phi) \|\Phi\| L(A_{j-1}),
 \end{aligned}
\]
which by recurrence gives $L(A_j) \le \big( (t+2j) \deg(\Phi)
\|\Phi\| \big)^j$.  The conclusion follows using $|A_j(\xi)| \le
L(A_j) \max(1,|\xi|)^{\deg(A_j)}$.
\end{proof}

\begin{lemma}
 \label{first:lemma:P0}
Let $n, t \in\bN^*$ with $1\le t\le n$, and let $P\in\ZT$ be a
non-zero polynomial of degree at most $n$.  Suppose that $P$ factors
as a product $P(T) = P_0(T) T^r \Phi(T)^t$ where $P_0 \in \ZT$,
$r\in \bN$ and $\Phi\in\ZT$, with $\Phi$ cyclotomic. Then, for each
$\xi\in\Cmult$ with $\Phi(\xi)\neq 0$, we have
\[
 \max_{0\le j< 2t-1} |P_0^{[j]}(\xi)|
 \le
 e^{10n} \max(|\xi|,|\xi|^{-1})^{3n}
 \min(1,|\Phi(\xi)|)^{-3t}
 \max_{0\le j< 2t-1} |P^{[j]}(\xi)|.
\]
\end{lemma}

\begin{proof}
Since $P_0(T) = P(T) T^{-r} \Phi(T)^{-t}$, Leibniz' formula for the
derivative of a product gives, for each integer $j\ge 0$,
\begin{equation}
 \label{first:lemma:P0:eq1}
 P_0^{[j]}(T)
 = \sum_{j_0+j_1+j_2=j} P^{[j_0]}(T)\, (T^{-r})^{[j_1]}\,
 (\Phi(T)^{-t})^{[j_2]},
\end{equation}
where the summation runs through all decompositions of $j$ as a sum
of non-negative integers $j_0,j_1,j_2$.  Let $\xi\in\Cmult$ with
$\Phi(\xi)\neq 0$.  As we have $r\le n$ and $t\le n$, we find, for
each $j=0,1,\dots,2t$,
\[
 \left| \left( T^{-r} \right)^{[j]}(\xi) \right|
 = \binom{r+j-1}{j} |\xi|^{-r-j}
 \le \frac{(3n)^j}{j!} \max(1,|\xi|^{-1})^{3n}.
\]
Since $\Phi^t$ divides $P$, we have $\deg(\Phi)\le n/t$, and since
$\Phi$ is monic with all of its roots on the unit circle, we deduce
that $\|\Phi\| \le 2^{\deg(\Phi)} \le 2^{n/t}$.    Then, for
$j=0,1,\dots,2t$, Lemma \ref{first:lemma:F^(-r)} gives
\[
 \left| \left( \Phi^{-t} \right)^{[j]}(\xi) \right|
 \le
 \frac{(5n)^j}{j!}\, 2^{2n} \max(1,|\xi|)^{2n} \min(1,|\Phi(\xi)|)^{-3t}.
\]
Combining these estimates with \eqref{first:lemma:P0:eq1}, we
conclude that
\[
 \max_{0\le j< 2t-1} |P_0^{[j]}(\xi)|
 \le
 C \max(|\xi|,|\xi|^{-1})^{3n} \min(1,|\Phi(\xi)|)^{-3t}
   \max_{0\le j< 2t-1} |P^{[j]}(\xi)|,
\]
with
\[
 C
 = \sum_{j_1,j_2 \ge 0} \frac{(3n)^{j_1}(5n)^{j_2}}{j_1!j_2!} 2^{2n}
 \le e^{10n}.
\]
\end{proof}

\begin{lemma}
 \label{first:lemma:derivatives}
Let $a,t \in \bN^*$, $P\in\ZT$ and $F(T)=P(T^a)$. For each
$\xi\in\bC$, we have
\[
 \max_{0\le j< t} |F^{[j]}(\xi)|
 \le
 (2+|\xi|)^{at} \max_{0\le j< t} |P^{[j]}(\xi^a)|.
\]
\end{lemma}

\begin{proof}
Let $n=\deg(P)$.  Expanding $F$ and $P$ in Taylor series around
$\xi$ and $\xi^a$ respectively, we find
\[
 \sum_{j=0}^{an} F^{[j]}(\xi) T^j
 = F(T+\xi)
 = P\big((T+\xi)^a\big)
 = \sum_{j=0}^n P^{[j]}(\xi^a) \big((T+\xi)^a-\xi^a\big)^j.
\]
Since $T^t$ divides $\big((T+\xi)^a-\xi^a\big)^j$ for each $j\ge t$,
this shows that the polynomials
\[
 \sum_{j=0}^{t-1} F^{[j]}(\xi) T^j
 \et
 \sum_{j=0}^{t-1} P^{[j]}(\xi^a) \big((T+\xi)^a-\xi^a\big)^j
\]
have the same coefficients of $T^j$ for $j=0,1,\dots,t-1$.
Therefore the length of the first is bounded above by that of the
second, and so we obtain
\[
 \sum_{j=0}^{t-1} |F^{[j]}(\xi)|
 \le
 \sum_{j=0}^{t-1} |P^{[j]}(\xi^a)| (1+|\xi|)^{aj}
 \le (2+|\xi|)^{at} \max_{0\le j< t} |P^{[j]}(\xi^a)|.
\]
\end{proof}

\begin{proof}[Proof of Proposition \ref{first:prop}]
Fix temporarily a choice of $a\in A$, $\xi\in E$ and $k\in \bN$ with
$k<t$, and put $\tP = P_0^{[k]}(T^a)$.  Since $P_0$ divides $P$ and
since $4n\le \log X$ by \eqref{first:prop:eq1}, we find
\begin{equation}
 \label{first:proofprop:eq1}
 \deg(\tP)\le a\deg(P_0)\le Mn
 \et
 H(\tP) \le 2^n H(P_0) \le 2^n e^n X \le X^{3/2}.
\end{equation}
According to Lemma \ref{first:lemma:derivatives}, we have
\[
 \max_{0\le j <t} |\tP^{[j]}(\xi)|
 \le (2+|\xi|)^{at} \max_{0\le j <t} |P_0^{[k][j]}(\xi^a)|
 \le (2+c_E)^{Mt} 2^{2t} \max_{0\le j <2t-1} |P_0^{[j]}(\xi^a)|.
\]
By Lemma \ref{first:lemma:P0}, we also have
\[
 \begin{aligned}
  \max_{0\le j< 2t-1} |P_0^{[j]}(\xi^a)|
  &\le
    e^{10n} \max(|\xi^a|,|\xi^a|^{-1})^{3n}
     \min(1,|\Phi(\xi^a)|)^{-3t}
     \max_{0\le j< 2t-1} |P^{[j]}(\xi^a)| \\
  &\le e^{10n} c_E^{3Mn} \min(1,\delta_\Phi)^{-3t} \delta_P.
 \end{aligned}
\]
Combining the last two estimates and using $t\le n\le Mn$ and $e\le
2+c_E$, we obtain
\begin{equation}
 \label{first:proofprop:eq2}
 \max_{0\le j <t} |\tP^{[j]}(\xi)|
  \le (2+c_E)^{16Mn} \min(1,\delta_\Phi)^{-3t} \delta_P.
\end{equation}
With the estimates \eqref{first:proofprop:eq1} and
\eqref{first:proofprop:eq2} at hand, we are now ready to apply
Proposition \ref{prelim:prop:resultant} to the collection of
polynomials $P_0^{[k]}(T^a)$ with $a\in A$ and $0\le k< t$.  Using
the hypotheses \eqref{first:prop:eq1}, it gives
\[
 \begin{aligned}
 \prod_{\xi\in E} \left( \frac{|Q(\xi)|}{\cont(Q)} \right)^t
 &\le
   e^{10(Mn)^2}
   (2+c_E)^{4(Mn)t|E|}
   \Delta_E^{-t^2}
   (X^{3/2})^{2Mn}
   \left( \frac{(2+c_E)^{16Mn} \delta_P}{\min(1,\delta_\Phi)^{3t}}
         \right)^{t|E|}\\
 &\le
   X^{5Mn}
   \Delta_E^{-t^2}
   \left( \frac{\delta_P}{\min(1,\delta_\Phi)^{3t}}
         \right)^{t|E|}.
 \end{aligned}
\]
\end{proof}

%
%
\section{Cyclotomic polynomials}
\label{sec:cyclo}

In order to apply Proposition \ref{first:prop} to the proof of our
main Theorem \ref{intro:thm1}, we need a lower bound for the
absolute value of a cyclotomic polynomial on an appropriate subset
of a finitely generated subgroup of $\Cmult$.  When the generators
of that subgroup do not all have absolute value one, the required
estimate is easy to derive.  The reader who wants a proof of Theorem
\ref{intro:thm1} under this simplifying assumption can skip this
section and go directly to the last proposition of the next section
where a suitable estimate is proved.

For the rest of this section, we fix a positive integer $m$ and
non-zero complex numbers $\xi_1,\dots,\xi_m$.  For each $m$-tuple of
integers $\ui = (i_1,\dots,i_m)$, we write for shortness $\uxi^\ui =
\xi_1^{i_1} \cdots \xi_m^{i_m}$, and we define $\|\ui\| =
\max\{|i_1|, \dots, |i_m|\}$ to be the maximum norm of $\ui$. Our
goal is to prove the following result dealing with values of
cyclotomic polynomials at the points $\uxi^\ui$.

\begin{proposition}
 \label{cyclo:prop}
Let $d,N\in\bN^*$ and $\delta\in\bR$ with
\begin{equation}
 \label{cyclo:prop:eq1}
 0 < \delta \le (8md^4N)^{-2md},
\end{equation}
and let $\Phi\in\ZT$ be a cyclotomic polynomial of degree $\le d$.
Then, there exist relatively prime positive integers $a_1, \dots,
a_m, D$ with $D\le (2md^2N)^m$ such that, upon defining
\[
 L(i_1,\dots,i_m)=a_1i_1+\cdots+a_mi_m
\]
for each $(i_1,\dots,i_m)\in\bZ^m$, at least one of the following
conditions holds:
\begin{itemize}
\item[1)] There exists a proper subspace $U$ of\/ $\bQ^m$ such that
 we have $|\Phi(\uxi^\ui)| \ge \delta$ for any point $\ui\in
 \bZ^m$ with $\ui\notin U$, $\|\ui\|\le N$ and
 $\gcd(L(\ui),D)=1$.
\item[2)] There exists a root $Z$ of $\Phi$ which is a root of unity
 of order exactly $D$ such that, upon denoting by $G$ the
 multiplicity of $Z$ as a root of $\Phi$, we have
 $|\uxi^\ui-Z^{L(\ui)}|^G \le
 \delta^{1/2}$ for each $\ui\in\bZ^m$ with $\|\ui\|\le N$.
\end{itemize}
\end{proposition}

When the condition 2) does not hold, the condition 1) necessarily
holds and provides the kind of estimate that we are looking for.
This happens for example when $\xi_1,\dots,\xi_m$ do not all have
absolute value one and when $N$ is sufficiently large in terms of
$\xi_1,\dots,\xi_m$, because under the condition 2) we find, for
each $j=1,\dots,m$,
\[
 \big| |\xi_j| - 1 \big|
 \le | \xi_j - Z^{a_j} |
 \le \delta^{1/(2G)}
 \le \delta^{1/(2d)}
 \le (8md^4N)^{-m}.
\]
In the next section we carry an independent analysis of this
situation (see Proposition \ref{reduction:prop2}).  We also show
that the condition 2) cannot hold for $N$ sufficiently large when
$\xi_1,\dots,\xi_m$ are as in the statement of our main theorem,
with $m\ge 2$.

Before going into the proof of Proposition \ref{cyclo:prop}, we also
note that the conditions 1) and 2) are almost mutually exclusive in
the following sense.  Suppose that the condition 2) holds, and let
$\ui$ be any point of $\bZ^m$ satisfying $\|\ui\|\le N$ and
$\gcd(L(\ui),D)=1$. Then, we have $|\uxi^\ui| \le 1+\delta^{1/(2G)}
\le 2$, and $Z^{L(\ui)}$ is a conjugate of $Z$ over $\bQ$.  So the
latter is also a root of $\Phi$ of multiplicity $G$.  Upon writing
$\Phi(T) = \Psi(T) (T-Z^{L(\ui)})^G$ with $\Psi\in\CT$, we find that
$|\Psi(\uxi^\ui)| \le (|\uxi^\ui|+1)^d \le 3^d$ (since $\Psi$ is
monic of degree at most $d$ with all its roots of absolute value
one), and thus $|\Phi(\uxi^\ui)| \le 3^{d} \delta^{1/2}$.

The proof of Proposition \ref{cyclo:prop} requires several lemmas
about cyclotomic polynomials and their roots.  The first three of
them are quite general.

\begin{lemma}
 \label{cyclo:lemma1}
Let $d \in \bN^*$, let $\Phi\in\ZT$ be a cyclotomic polynomial of
degree at most $d$, and let $\zeta$ be a root of $\Phi$.  Denote by
$\ell$ the order of $\zeta$ as a root of unity, and by $g$ its
multiplicity as a root of $\Phi$. Then, we have
\begin{equation}
 \label{cyclo:lemma1:eq1}
 \ell \le \frac{2d\log_2(2d)}{g} \le 2d^2,
\end{equation}
where $\log_2$ stands for the logarithm in base $2$.
\end{lemma}

\begin{proof}
The theory of cyclotomic fields gives $[\bQ(\zeta):\bQ] =
\phi(\ell)$ where $\phi$ denotes Euler's totient function. Since
$\zeta$ is a root of $\Phi$ of multiplicity $g$, this implies that
$g\phi(\ell) \le d$.  Putting $k=\omega(\ell)+1$ where
$\omega(\ell)$ denotes the number of distinct prime factors of
$\ell$, we have $k\ge 1$,
\[
 \phi(\ell)
 = \ell \prod_{p|\ell} \Big(1-\frac{1}{p}\Big)
 \ge \ell \prod_{i=2}^k \Big(1-\frac{1}{i}\Big)
 = \frac{\ell}{k}
 \et
 \ell \ge \prod_{p|\ell} p \ge k!.
\]
Since $k! \ge k^2/2$, this gives $k\le \sqrt{2\ell}$, so $\phi(\ell)
\ge \sqrt{\ell/2}$, and thus $\ell \le 2\phi(\ell)^2$.  Since $k!
\ge 2^{k-1}$, we also find $k\le 1+\log_2(\ell)$ which combined with
the previous upper bound for $\ell$ gives $k \le
2\log_2(2\phi(\ell))$.  Since $\phi(\ell) \le d/g \le d$, we
conclude that $k \le 2\log_2(2d)$ and consequently $\ell \le k
\phi(\ell) \le 2(d/g)\log_2(2d)$.
\end{proof}

For roots of unity, Liouville's inequality takes a very simple form:

\begin{lemma}
 \label{cyclo:lemma2}
Let $\zeta_1$ and $\zeta_2$ be two distinct roots of unity with
respective orders $\ell_1$ and $\ell_2$. Then, we have
\begin{equation*}
 |\zeta_1-\zeta_2| \ge \frac{4}{\ell_1\ell_2}.
\end{equation*}
\end{lemma}

\begin{proof}
For $j=1,2$, write $\zeta_j=\exp(2\pi r_j \sqrt{-1})$ where $r_j$ is
a rational number with denominator $\ell_j$.  Upon subtracting from
$r_1$ a suitable integer, we can arrange that $|r_1-r_2|\le 1/2$.
Since $|\exp(t \sqrt{-1})-1| \ge 2|t|/\pi$ for any real number $t$
with $|t|\le \pi$, we deduce that
\[
 |\zeta_1-\zeta_2|
  = |\exp(2 \pi (r_1-r_2) \sqrt{-1} )-1| \ge 4|r_1-r_2|.
\]
Since $r_1-r_2$ is a non-zero rational number with denominator
dividing $\ell_1\ell_2$, we also have $|r_1-r_2| \ge
(\ell_1\ell_2)^{-1}$ and the conclusion follows.
\end{proof}

\begin{lemma}
 \label{cyclo:lemma3}
Let $d\in\bN^*$ and let $\Phi\in\ZT$ be a cyclotomic polynomial of
degree at most $d$.  For any $\xi\in\bC$, there exists a root
$\zeta$ of $\Phi$ with
\begin{equation}
 \label{cyclo:lemma3:eq1}
 |\xi-\zeta|^g \le (2d^4)^d |\Phi(\xi)|,
\end{equation}
where $g$ denotes the multiplicity of $\zeta$ as a root of $\Phi$.
\end{lemma}

\begin{proof}
Let $\zeta$ be a root of $\Phi$ which is closest to $\xi$, and let
$g$ be its multiplicity.  Since $\Phi$ is monic, we can write
$\Phi(T) = (T-\zeta_1)\cdots(T-\zeta_s)$ where $s\le d$ is the
degree of $\Phi$ and where $\zeta_1,\dots,\zeta_s$ are roots of
unity with $\zeta_1=\cdots=\zeta_g=\zeta$.  By Lemma
\ref{cyclo:lemma1}, each $\zeta_j$ has order at most $2d^2$.  Thus,
for $j=g+1,\dots, s$, Lemma \ref{cyclo:lemma2} gives
$|\zeta-\zeta_j| \ge d^{-4}$. For the same values of $j$ we also
have $|\zeta-\zeta_j| \le |\xi-\zeta| + |\xi-\zeta_j| \le 2
|\xi-\zeta_j|$ by virtue of the choice of $\zeta$, and so
$|\xi-\zeta_j| \ge (2d^4)^{-1}$. This gives $|\Phi(\xi)| \ge
|\xi-\zeta|^g (2d^4)^{g-s} \ge |\xi-\zeta|^g (2d^4)^{-d}$.
\end{proof}

The last lemma is more technical and provides the key to the proof
of Proposition \ref{cyclo:prop}.

\begin{lemma}
 \label{cyclo:lemma4}
Let $\ell,N\in\bN^*$ and $\rho\in\bR$ with $0 < \rho \le
(1/2)(mN)^{-m}$.  Suppose that there exist linearly independent
points $\ui^{(1)},\dots,\ui^{(m)}$ of $\bZ^m$ of norm at most $N$,
and roots of unity $\zeta_1,\dots,\zeta_m$ of order at most $\ell$
such that $|\uxi^{\ui^{(k)}}-\zeta_k| \le \rho$ for $k=1,\dots,m$.
Then, there exist an integer $D$ with $1\le D \le (\ell m N)^m$, a
root of unity $Z$ of order $D$, and non-zero integers $a_1, \dots,
a_m$ with $\gcd(a_1, \dots, a_m, D) = 1$ such that, for each $\ui =
(i_1,\dots,i_m) \in \bZ^m$ with norm $\|\ui\| \le N$, we have
\begin{equation}
 \label{cyclo:lemma4:eq1}
 |\uxi^\ui-Z^{a_1i_1+\cdots+a_mi_m}| \le 4 (mN)^m \rho.
\end{equation}
\end{lemma}

\begin{proof} For $k=1,\dots,m$, we can write $\uxi^{\ui^{(k)}} =
\zeta_k (1 + \rho_k)$ for a complex number $\rho_k$ with
$|\rho_k|\le \rho$.  Put $\rho'_k = \log (1+\rho_k) =
-\sum_{j=1}^\infty (-\rho_k)^j/j$.  Since $|\rho_k| \le 1/2$, we
find $|\rho'_k| \le 2|\rho_k|$, and so
\begin{equation}
 \label{cyclo:lemma4:eq2}
 \uxi^{\ui^{(k)}} = \zeta_k \exp(\rho'_k)
 \quad \text{with} \quad
 |\rho'_k|\le 2\rho.
\end{equation}
Let $M$ be the square $m\times m$ matrix whose rows are
$\ui^{(1)},\dots,\ui^{(m)}$.  For $j=1,\dots,m$, let
$(b_{j1},\dots,b_{jm})$ denote the $j$-th row of the adjoint of $M$,
and let $\ue_j$ denote the $j$-th row of the $m\times m$ identity
matrix.  Since $\det(M)\ue_j = b_{j1}\ui^{(1)} + \cdots +
b_{jm}\ui^{(m)}$, we find by \eqref{cyclo:lemma4:eq2}
\begin{equation}
 \label{cyclo:lemma4:eq3}
 \xi_j^{\det(M)}
  = \zeta_1^{b_{j1}}\cdots\zeta_m^{b_{jm}}
    \exp\Big( \sum_{k=1}^m  b_{jk}\rho'_k \Big).
\end{equation}
Since $|b_{jk}| \le (m-1)!N^{m-1}$ for $k=1,\dots,m$ and since
$\det(M)$ is a non-zero integer, we deduce from
\eqref{cyclo:lemma4:eq2} and \eqref{cyclo:lemma4:eq3} that
\begin{equation}
 \label{cyclo:lemma4:eq4}
 \xi_j
  = Z_j\exp(\rho''_j)
 \quad \text{with} \quad
 Z_j^{\det(M)} = \zeta_1^{b_{j1}}\cdots\zeta_m^{b_{jm}}
 \quad \text{and} \quad
 |\rho''_j| \le 2m!N^{m-1}\rho.
\end{equation}
Let $Z$ be a generator of the subgroup of $\Ctor$ spanned by
$Z_1,\dots,Z_m$, and let $D$ be the order of $Z$.  Since
$Z^{\det(M)}$ belongs to the subgroup spanned by $\zeta_1, \dots,
\zeta_m$ and since the latter have order at most $\ell$, the order
$D$ of $Z$ is at most $\ell^m|\det(M)| \le (\ell m N)^m$.  For
$j=1,\dots,m$, we choose an integer $a_j\ge 1$ such that
$Z_j=Z^{a_j}$.  Then, because of the choice of $Z$, we have
$\gcd(a_1,\dots,a_m,D)=1$, and for each $\ui=(i_1,\dots,i_m)\in
\bZ^m$ with $\|\ui\| \le N$ we find by \eqref{cyclo:lemma4:eq4}
\begin{equation}
 \label{cyclo:lemma4:eq5}
 \uxi^{\ui}
  = Z^{a_1i_1+\cdots+a_mi_m} \exp(\rho'''_\ui)
 \quad \text{with} \quad
 |\rho'''_\ui| \le (mN)(2m!N^{m-1}\rho) \le 2(mN)^m\rho.
\end{equation}
Since $|\rho'''_\ui| \le 1$, we also have $|\exp(\rho'''_\ui)-1| \le
2|\rho'''_\ui|$ and so \eqref{cyclo:lemma4:eq1} follows.
\end{proof}

\begin{proof}[Proof of Proposition \ref{cyclo:prop}]
Let $I_N$ denote the set of points $\ui\in \bZ^m$ with $\|\ui\| \le
N$, and let $I_{N,\Phi}$ denote the set of points $\ui\in I_N$ such
that $|\Phi(\uxi^\ui)| < \delta$.  If $I_{N,\Phi}$ is contained in a
proper subspace $U$ of $\bQ^m$, we are done.  Assume the contrary.
Then, since $I_{N,\Phi}$ is a finite set, there exists a smallest
positive real number $\rho$ for which it contains $m$ linearly
independent points $\ui^{(1)},\dots,\ui^{(m)}$ with the property
that each of the complex numbers $\uxi^{\ui^{(1)}}, \dots,
\uxi^{\ui^{(m)}}$ is at a distance $\le \rho$ from a zero of $\Phi$.
Lemma \ref{cyclo:lemma3} shows that, for each $\ui\in I_{N,\Phi}$,
there exists a root $\zeta$ of $\Phi$ with
\begin{equation}
 \label{cyclo:proof_prop:eq1}
 |\uxi^\ui-\zeta|^g \le (2d^4)^d \delta
\end{equation}
where $g$ denotes the multiplicity of $\zeta$.  Since $g\le d$, this
implies that $\rho \le 2d^4 \delta^{1/d}$.  Since the hypothesis
\eqref{cyclo:prop:eq1} gives $2d^4 \delta^{1/d} \le (2mN)^{-m}$ and
since, by Lemma \ref{cyclo:lemma1}, any root $\zeta$ of $\Phi$ has
order $\le 2d^2$, Lemma \ref{cyclo:lemma4} provides us with
relatively prime positive integers $a_1, \dots, a_m, D$ with $D \le
(2d^2 m N)^m$, and a root of unity $Z$ of order $D$, such that for
each $\ui = (i_1,\dots,i_m) \in I_N$, we have
\begin{equation}
 \label{cyclo:proof_prop:eq2}
 |\uxi^\ui - Z^{L(\ui)}| \le 4 (m N)^m \rho,
 \quad
 \text{where $L(\ui)=a_1i_1+\cdots+a_mi_m$.}
\end{equation}
If we choose $\ui\in I_{N,\Phi}$ and if $\zeta$ is a root of $\Phi$
satisfying \eqref{cyclo:proof_prop:eq1}, this gives
\begin{equation}
 \label{cyclo:proof_prop:eq3}
 |\zeta-Z^{L(\ui)}|
  \le (1 + 4(m N)^m ) 2d^4 \delta^{1/d}.
\end{equation}
As $\zeta$ and $Z^{L(\ui)}$ are roots of unity of order at most
$2d^2$ and $D$ respectively and since by \eqref{cyclo:prop:eq1} the
right hand side of \eqref{cyclo:proof_prop:eq3} is at most $4d^4
(4mN)^m (8md^4N)^{-2m} < 4 D^{-1} (2d^2)^{-1}$, we conclude, by
Lemma \ref{cyclo:lemma2}, that both roots of unity are equal.
Therefore, $Z^{L(\ui)} = \zeta$ is a root of $\Phi$ when $\ui\in
I_{N,\Phi}$.

Finally, let $I_{N,\Phi,D}$ denote the set of points $\ui \in
I_{N,\Phi}$ with $\gcd(L(\ui), D) = 1$.  Again, if this set is
contained in a proper subspace of $\bQ^m$, the first condition of
the proposition holds. Suppose on the contrary that $I_{N,\Phi,D}$
contains $m$ linearly independent points.  For each $\ui\in
I_{N,\Phi,D}$, the root of unity $Z^{L(\ui)}$ is a conjugate of $Z$
over $\bQ$, so it is a root of $\Phi$ of the same multiplicity $G$
as $Z$, and the inequality \eqref{cyclo:proof_prop:eq1} gives
$|\uxi^\ui - Z^{L(\ui)}|^G \le (2d^4)^d \delta$.  As $I_{N,\Phi,D}$
contains $m$ linearly independent points, this means that $\rho^G
\le (2d^4)^d \delta$. By \eqref{cyclo:proof_prop:eq2} and the fact
that $G\le d$, we conclude that, for each $\ui\in I_N$, we have
\[
 |\uxi^\ui-Z^{L(\ui)}|^G
  \le \big(4(mN)^m\rho\big)^G
  \le (4mN)^{md} (2d^4)^d \delta
  \le \delta^{1/2}.
\]
\end{proof}

%
%
\section{Avoiding cyclotomic factors in rank at least two}
\label{sec:reduction}

In this section, we consider two instances where only the first
alternative in Proposition \ref{cyclo:prop} holds.  As observed in
the preceding section, the simplest case is when $\xi_1, \dots,
\xi_m$ do not all have absolute value one. The reader who wants to
restrict to this situation can go directly to Proposition
\ref{reduction:prop2}, where a short independent proof is given, and
omit the rest of the section. The second case is when $\xi_1, \dots,
\xi_m$ are multiplicatively independent with $m\ge 2$, and generate
over $\bQ$ a field of transcendence degree one.  To show that the
latter condition is sufficient, we first establish the following
measure of simultaneous approximation by roots of unity, where
$\phi$ stands for the Euler totient function.

\begin{proposition}
 \label{reduction:prop}
Let $m\ge 2$ be an integer, and let $\xi_1,\dots,\xi_m\in \Cmult$ be
multiplicatively independent non-zero complex numbers which generate
over $\bQ$ a field of transcendence degree one. For any choice of
positive integers $a_1,\dots,a_m,D$ and for any root of unity
$Z\in\Ctor$ of order $D$, we have
\begin{equation}
 \label{reduction:prop:eq1}
 \max_{1\le j\le m} |\xi_j-Z^{a_j}| > c^{\phi(D)}
\end{equation}
where $c$ is a constant depending only on $\xi_1, \dots, \xi_m$ with
$0<c \le 1$.
\end{proposition}

In the proof below as well as in the rest of the section, we use the
same notation as in Section \ref{sec:cyclo}. Namely, we denote by
$\|\ui\|$ the maximum norm of an integer point $\ui =
(i_1,\dots,i_m) \in \bZ^m$, and we define $\uxi^\ui = \xi_1^{i_1}
\cdots \xi_m^{i_m}$.

\begin{proof}
The field $R=\bQ(\xi_1,\dots,\xi_m)$ is a field of functions in one
variable over $\bQ$ (see Chapter 1 of \cite{Ch}).  Let $K$ denote
its field of constants and, for $j=1,\dots,m$, let $\gb_j$ denote
the divisor of poles of $\xi_j$. Let $J$ be the ideal of polynomials
of $\bQ[T_1,\dots,T_m]$ which vanish at the point
$(\xi_1,\dots,\xi_m)$, and let $P_1,\dots,P_s$ be a system of
generators of this ideal, chosen in $\bZ[T_1,\dots,T_m]$.  Define
\[
 c_1 = \max_{1\le k \le s} \left( L(P_k) \max_{1\le j\le m}
 (1+|\xi_j|)^{\deg(P_k)} \right)
 \et
 c_2 = [K:\bQ] \sum_{j=1}^m \deg(\gb_j),
\]
and choose a real number $c$ with $0 < c < c_1^{-1}$ such that
\eqref{reduction:prop:eq1} holds whenever $D\le (3c_2)^6$ (this
involves a finite number of inequalities).  We claim that, for such
a value of $c$, the estimate \eqref{reduction:prop:eq1} holds in
general.

To prove this, suppose on the contrary that there exist positive
integers $a_1,\dots,a_m,D$ and a root of unity $Z$ of order $D$
which satisfy
\[
 \max_{1\le j\le m} |\xi_j-Z^{a_j}|
 \le c^{\phi(D)}.
\]
Upon replacing $a_1,\dots,a_m$, $D$ and $Z$ respectively by
$a_1/a,\dots,a_m/a$, $D/a$ and $Z^a$ where $a =
\gcd(a_1,\dots,a_m,D)$, we may assume without loss of generality
that $a_1,\dots,a_m,D$ are relatively prime. For each $k=1,\dots,s$,
the norm of $P_k(Z^{a_1},\dots,Z^{a_m})$ from $\bQ(Z)$ to $\bQ$ is
an integer given by
\[
 \mathrm{N}_{\bQ(Z)/\bQ}\big( P_k(Z^{a_1},\dots,Z^{a_m}) \big)
 =
 \prod_{\substack{1\le j\le D \\ \gcd(j,D)=1}}
   P_k(Z^{ja_1},\dots,Z^{ja_m}).
\]
Since $
 \big| P_k(Z^{a_1},\dots,Z^{a_m}) \big|
 =
 \big| P_k(\xi_1,\dots,\xi_m) - P_k(Z^{a_1},\dots,Z^{a_m}) \big|
 \le
 c_1 \max_{1\le j\le m} |\xi_j-Z^{a_j}|
$, and since $\big| P_k(Z^{ja_1},\dots,Z^{ja_m}) \big| \le L(P_k)
\le c_1$ for each integer $j$, we deduce that
\[
 \left|
 \mathrm{N}_{\bQ(Z)/\bQ}\big( P_k(Z^{a_1},\dots,Z^{a_m}) \big)
 \right|
 \le c_1^{\phi(D)} \max_{1\le j\le m} |\xi_j-Z^{a_j}|
 \le (c_1c)^{\phi(D)}
 < 1.
\]
Thus the norm of $P_k(Z^{a_1},\dots,Z^{a_m})$ is $0$ and so we have
$P_k(Z^{a_1},\dots,Z^{a_m})=0$ for $k=1,\dots,s$.  According to
\cite[Ch.~1, \S4, Cor.~1]{Ch}, this implies the existence of a place
$\gp$ of $R$ which is a common zero of
$\xi_1-Z^{a_1},\dots,\xi_m-Z^{a_m}$. The residue field of this place
contains $\bQ[Z^{a_1},\dots,Z^{a_m}] $ which is simply $\bQ(Z)$
since $\gcd(a_1,\dots,a_m,D) = 1$.  Thus, we have
\begin{equation}
 \label{reduction:prop:proof:eq1}
 [K:\bQ] \deg(\gp) \ge [\bQ(Z):\bQ] = \phi(D).
\end{equation}

Define $L(\ui) = a_1i_1 + \cdots + a_mi_m$ for each $\ui =
(i_1,\dots,i_m) \in \bZ^m$, and choose any non-zero point $\ui \in
\bZ^m$ such that $L(\ui)\equiv 0 \mod D$.  Since $\xi_1,\dots,\xi_m$
are multiplicatively independent, the difference $\eta = \uxi^\ui-1$
is a non-zero element of $R$.  Let $\ga$ denote its divisor of
zeros, and $\gb$ its divisor of poles. Then $\ga$ and $\gb$ have the
same degree. Similarly, for each $j=1,\dots,m$ the divisor of zeros
$\ga_j$ of $\xi_j$ has the same degree as its divisor of poles
$\gb_j$.  Since $\gb$ is also the divisor of poles of $\uxi^\ui =
\xi_1^{i_1} \cdots \xi_m^{i_m}$, we deduce that
\[
 \deg(\gb) \le \|\ui\| \sum_{j=1}^m \deg(\gb_j).
\]
On the other hand, since $Z^{L(\ui)}=1$, the place $\gp$ is a zero
of $\eta$ and so we have $\deg(\ga)\ge \deg(\gp)$. Combining this
with \eqref{reduction:prop:proof:eq1} and the above inequality, we
conclude that
\[
 \phi(D)
  \le [K:\bQ]\deg(\gp)
  \le [K:\bQ]\deg(\ga)
    = [K:\bQ]\deg(\gb)
  \le c_2 \|\ui\|.
\]
This observation implies that the function $f\colon\bZ^m \to
\bZ/D\bZ$ given by $f(\ui)=L(\ui)+D\bZ$ ($\ui\in\bZ^m$) is injective
on the set of points $\ui\in\bN^m$ with $\|\ui\| < c_2^{-1}\phi(D)$,
and therefore we have
\begin{equation}
 \label{reduction:prop:proof:eq2}
 D \ge (c_2^{-1}\phi(D))^m \ge (c_2^{-1}\phi(D))^2.
\end{equation}
On the other hand, since $Z$ is a root of a cyclotomic polynomial of
degree $\phi(D)$, Lemma \ref{cyclo:lemma1} gives $D \le
2\phi(D)\log_2(2\phi(D)) \le 3\phi(D)\log(2\phi(D))$. Since
$\log(x)\le \sqrt{x}$ for any positive real number $x$, this leads
to $D\le (3\phi(D))^{3/2}$ which combined with
\eqref{reduction:prop:proof:eq2} gives $D\le (3c_2)^6$. This is a
contradiction since we chose $c$ so that \eqref{reduction:prop:eq1}
holds for such a value of $D$.
\end{proof}

Combining the above result with Proposition \ref{cyclo:prop}, we
obtain:

\begin{corollary}
 \label{reduction:cor}
Let $m$, $\xi_1,\dots,\xi_m$ and $c$ be as in the statement of
Proposition \ref{reduction:prop}. Let $d,N\in\bN^*$ and
$\delta\in\bR$ with
\begin{equation}
 \label{reduction:cor:eq}
 0 < \delta \le \min\{ (8md^4N)^{-m},\, c \}^{2d},
\end{equation}
and let $\Phi\in\ZT$ be a cyclotomic polynomial of degree $\le d$.
Then, there exist relatively prime positive integers $a_1, \dots,
a_m, D$ with $D\le (2md^2N)^m$ and a proper subspace $U$ of\/
$\bQ^m$ such that we have $|\Phi(\uxi^\ui)| \ge \delta$ for any
point $\ui = (i_1,\dots,i_m) \in \bZ^m \setminus U$ with $\|\ui\|\le
N$ and $\gcd(a_1i_1+\cdots+a_mi_m, D) =1$.
\end{corollary}

\begin{proof}
Let $Z$ be a root of $\Phi$, let $D$ denote its order as a root of
unity, and let $G$ denote its multiplicity as a root of $\Phi$.
Since $d \ge \deg(\Phi) \ge G\phi(D)$, Proposition
\ref{reduction:prop} gives
\[
 \max_{1\le j\le m} |\xi_j-Z^{a_j}|^G
 > c^{G\phi(D)}
 \ge c^d
 \ge \delta^{1/2}
\]
for any choice of positive integers $a_1,\dots,a_m$.   The
conclusion follows by Proposition \ref{cyclo:prop}.
\end{proof}

The next result provides a substitute to Corollary
\ref{reduction:cor} when $\xi_1,\dots,\xi_m$ do not all have
absolute value one.

\begin{proposition}
 \label{reduction:prop2}
Let $m\ge 2$ be an integer, let $\xi_1,\dots,\xi_m\in \Cmult$ be
non-zero complex numbers not all of absolute value $1$, and let $N$
be a positive integer.  If $N$ is sufficiently large, there exists a
proper subspace $U$ of $\bQ^m$ such that we have $|\Phi(\uxi^\ui)|
\ge (8mN)^{-md}$ for each positive integer $d$, each cyclotomic
polynomial $\Phi\in\ZT$ of degree $\le d$, and each point
$\ui\in\bZ^m\setminus U$ with $\|\ui\|\le N$.
\end{proposition}

\begin{proof}
Write $\xi_j=\exp(u_j+v_j\sqrt{-1})$ with $u_j,v_j\in\bR$, for
$j=1,\dots,m$.  Then, $u_1,\dots,u_m$ are not all zero, and by a
result of Dirichlet (see for example \cite[Ch.~II, Thm 1A]{Sc}),
there exist integers $a_1,\dots,a_m$ and $b$ satisfying $1\le b\le
(2mN)^m$ and $|bu_j-a_j|\le (2mN)^{-1}$ for $j=1,\dots,m$.  If $N$
is large enough, the integers $a_1,\dots,a_m$ are not all zero, and
so the equation $a_1x_1+\cdots+a_mx_m=0$ defines a proper subspace
$U$ of $\bQ^m$.  For any $\ui=(i_1,\dots,i_m)\in\bZ^m\setminus U$
with $\|\ui\|\le N$, we have $|a_1i_1+\cdots+a_mi_m|\ge 1$ and thus
\begin{align*}
 |u_1i_1+\cdots+u_mi_m|
  &\ge
  \frac{1}{b} \left| \sum_{j=1}^m a_ji_j \right|
  - \frac{1}{b} \sum_{j=1}^m |bu_j-a_j|\,|i_j| \\
  &\ge
  \frac{1}{b}
  - \frac{1}{b} m (2mN)^{-1} N
  = \frac{1}{2b}
  \ge (4mN)^{-m}.
\end{align*}
Since $|\exp(x)-1|\ge |x|/2$ for each $x\in\bR$ with $|x|\le 1/2$,
we deduce that for the same choice of $\ui$ and any root of unity
$\zeta\in\Ctor$, we have
\[
 | \uxi^\ui - \zeta |
  \ge \big| |\uxi^\ui| - 1 \big|
   = | \exp(u_1i_1+\cdots+u_mi_m) - 1 |
  \ge 1 - \exp(-(4mN)^{-m})
  \ge (8mN)^{-m}.
\]
Consequently, for any positive integer $d$ and any cyclotomic
polynomial $\Phi\in\ZT$ of degree $\le d$, we get $|\Phi(\uxi^\ui)|
\ge (8mN)^{-md}$.
\end{proof}

%
%
\section{Estimates for an intersection}
\label{sec:inter}

Throughout this section, we fix an abelian group $\bG$ with its
group law denoted multiplicatively, and we fix a finite set of prime
numbers $A$ with cardinality at least $2$.  We denote by $\Gtor$ the
torsion subgroup of $\bG$.  For each subset $E$ of $\bG$, we define
\[
 \cO(E) = \{ x^p \,;\, x\in E,\, p\in A \}.
\]
For a singleton $\{x\}$, we simply write $\cO(x)$ to denote
$\cO(\{x\})$.  Then, for any subset $E$ of $\bG$, we have $\cO(E) =
\cup_{x\in E} \cO(x)$.  For each $x\in\bG$ and each integer $k\ge
1$, we also define $C_k(x)$ to be the set of all elements $y$ of
$\bG$ which satisfy a relation of the form
\begin{equation}
 \label{inter:defCk}
 x^{p_1\cdots p_k} = y^{q_1\cdots q_k},
\end{equation}
for a choice of prime numbers $p_1,\dots,p_k,q_1,\dots,q_k$ in $A$
(not necessarily distinct).  We also define $C_0(x)=\{x\}$.  With
this notation, the main result of this section reads as follows:

\begin{proposition}
 \label{inter:prop}
Let $E$ and $F$ be finite non-empty subsets of $\bG$ with
$\cO(E)\subseteq F$ and $E\cap\Gtor=\emptyset$. Suppose that
\begin{equation}
 \label{inter:prop:eq1}
 |F| \le \frac{1}{2^{\ell+1}(\ell+1)!}\binom{|A|}{\ell+2}
\end{equation}
for some integer $\ell$ with $0\le \ell\le |A|-2$.  Then, there
exist an integer $r\ge 1$, a sequence of points $x_1,\dots,x_r$ of
$E$, and partitions $E = E_1 \amalg\cdots\amalg E_r$ and $F = F_1
\amalg\cdots\amalg F_r \amalg F_{r+1}$ of $E$ and $F$ which, for
$i=1,\dots,r$, satisfy
\[
 \text{a)}\ E_i \subseteq C_\ell(x_i),
 \qquad
 \text{b)}\ F_i \subseteq \cO(E_i),
 \qquad
 \text{c)}\ |F_i| \ge \frac{|A|-\ell}{2(\ell+1)} |E_i|.
\]
\end{proposition}

This result can be viewed as a generalization of Proposition 6.2 of
\cite{ixi} (see the remark at the end of this section for more
details on how to derive the latter from the former). Its proof will
follow the same general pattern, although additional difficulties
come into play due to the fact that $\bG$ may contain non-trivial
torsion elements.  To deal with these, we use several additional
notions.

First of all, we say that two elements $x$ and $y$ of $\bG$ are
\emph{$A$-equivalent} and we write $x\sim_A y$ if there exist finite
sequences $(p_1,\dots,p_k)$ and $(q_1,\dots,q_\ell)$ of elements of
$A$ such that
\begin{equation}
 \label{inter:equiv:eq}
 x^{p_1\cdots p_k} = y^{q_1\cdots q_\ell}.
\end{equation}
This defines an equivalence relation on $\bG$.  In view of the
preceding definitions, for any $x\in\bG$ and any integer $k\ge 0$,
the equivalence class of $x$ contains $C_k(x)$.

Fix a non-torsion element $x$ of $\bG$ and a point $y$ in the same
equivalence class.  Then, $y$ is also a non-torsion element of $G$.
Moreover, if $\gen{x}$ denotes the subgroup of $\bG$ generated by
$x$, then the set of integers $i$ such that $y^i\in\gen{x}$ is a
non-trivial subgroup of $\bZ$.  We define $\den_x(y)$ to be the
positive generator $n$ of this group.  Then, since $x$ is
non-torsion, there exists a unique integer $m$ such that $y^n=x^m$,
and we define $\num_x(y)=m$.  Note that these integers $m$ and $n$
may not be relatively prime, and therefore the fraction $m/n$ may
not be in reduced form. However, the following lemma shows useful
properties for these notions of logarithmic ``numerator'' and
``denominator'' of $y$ with respect to $x$.

\begin{lemma}
 \label{inter:lemma0}
Let $x$ and $y$ be non-torsion elements of $\bG$ in the same
equivalence class.  Put $n=\den_x(y)$ and $m=\num_x(y)$, and choose
elements $p_1, \dots, p_k, q_1, \dots, q_\ell$ of $A$ such that
\eqref{inter:equiv:eq} holds.  Then, $m$ (resp.~$n$) is a positive
divisor of $p_1\cdots p_k$ (resp.~$q_1\cdots q_\ell$), and we have
\begin{equation}
 \label{inter:lemma0:eq1}
 \frac{m}{n}
 = \frac{p_1\cdots p_k}{q_1\cdots q_\ell}.
\end{equation}
Moreover, if $q$ is an element of $A$ not dividing $n$, then the
point $z=y^q$ satisfies $\den_x(z)=n$ and $\num_x(z)=qm$.
\end{lemma}

\begin{proof}
Since $x\notin\Gtor$, the equality \eqref{inter:equiv:eq} combined
with $y^n=x^m$ leads to \eqref{inter:lemma0:eq1}.  Moreover, as
\eqref{inter:equiv:eq} gives $y^{q_1\cdots q_\ell} \in \gen{x}$, it
follows from the definition of $\den_x(y)$ that $n$ is a positive
divisor of $q_1\cdots q_\ell$.  Then, since all the elements of $A$
are positive, we deduce from \eqref{inter:lemma0:eq1} that $m$ is a
positive divisor of $p_1\cdots p_k$.  This proves the first part of
the lemma.

For the second part, fix a prime number $q\in A$ not dividing $n$.
Put $z=y^q$, $n' = \den_x(z)$ and $m' = \num_x(z)$.  Since $z^n =
y^{qn} = x^{qm} \in \gen{x}$, it follows, by definition of $n'$,
that $n'$ divides $n$.  Moreover, since $y^{qn'} = z^{n'} = x^{m'}
\in \gen{x}$, it also follows from the definition of $n$ that $n$
divides $qn'$.  Since, by hypothesis, $q$ and $n$ are relatively
prime, and since $n$ and $n'$ are positive, these two divisibility
relations imply that $n=n'$.  Then, since $x\notin\Gtor$, the
equality $x^{m'}=z^{n}=x^{qm}$ implies that $m'=qm$.
\end{proof}

For any integer $k\ge 0$, any non-torsion point $x$ of $\bG$ and any
subset $E$ of $\bG$, we define
\[
 C_k(x,E) = C_k(x) \cap E
 \et
 D_k(x,E) = \cO(C_k(x,E)).
\]
With this notation, the first part of Lemma \ref{inter:lemma0} shows
that, for each $y\in C_k(x,E)$ and each $z\in D_k(x,E)$, the
integers $\den_x(y)$, $\num_x(y)$ and $\den_x(z)$ are products of at
most $k$ elements of $A$, while $\num_x(z)$ is a product of at most
$k+1$ elements of $A$, counting multiplicities. We also note that if
a subset $F$ of $\bG$ contains $\cO(E)$, then it contains
$D_k(x,E)$.  The next lemma compares the sizes of $C_k(x,E)$ and
$D_k(x,E)$.

\begin{lemma}
 \label{inter:lemma1}
Let $E$ be a finite subset of $\bG$, let $k\ge 0$ be an integer, and
let $x\in\bG$ with $x\notin\Gtor$. Then, we have
\[
 |D_k(x,E)| \ge \frac{|A|-k}{k+1} |C_k(x,E)|.
\]
\end{lemma}

\begin{proof}
Put $C=C_k(x,E)$ and $D=\cO(C)$, so that $D=D_k(x,E)$.  We denote by
$N$ the set of all pairs $(y,q)\in C\times A$ such that $q$ divides
$\den_x(y)$, and we put $P=(C\times A)\setminus N$. Then, since $N$
and $P$ form a partition of $C\times A$, we have
\begin{equation}
 \label{inter:lemma1:eq1}
 |N|+|P| = |C|\,|A|.
\end{equation}

For any given $y\in C$, the integer $\den_x(y)$ is a product of at
most $k$ prime numbers (including multiplicities).  Therefore there
are at most $k$ distinct elements $q$ of $A$ such that $(y,q)\in N$.
This being true for each $y\in C$, we deduce that
\begin{equation}
 \label{inter:lemma1:eq2}
 |N|\le k|C|.
\end{equation}

Consider the surjective map $\varphi\colon C\times A \to D$ given by
$\varphi(y,q)=y^q$ for each $(y,q)\in C\times A$. We claim that, for
each $z\in D$, we have $|\varphi^{-1}(z)\cap P| \le k+1$. If we
admit this result, then we find
\[
 |P| = |\varphi^{-1}(D)\cap P| \le (k+1) |D|,
\]
and by combining this estimate with \eqref{inter:lemma1:eq1} and
\eqref{inter:lemma1:eq2}, we deduce that
\[
 (k+1)|D| \ge |P| = |A|\,|C| - |N| \ge (|A|-k)|C|,
\]
as announced.

To prove the above claim, suppose that $(y,q)\in \varphi^{-1}(z)\cap
P$ for some fixed $z\in D$. Put $n = \den_x(y)$ and $m=\num_x(y)$.
By hypothesis, we have $y^q=z$ and $q$ is prime to $n$.   According
to Lemma \ref{inter:lemma0}, this implies that $\den_x(z)=n$ and
$\num_x(z)=qm$.  So, $n$ is known (it depends only on $x$ and $z$)
and $q$ is a prime divisor of $\num_x(z)$.  Moreover, since $z\in
D$, the integer $\num_x(z)$ is a product of at most $k+1$ prime
numbers of $A$. So, this leaves at most $k+1$ possibilities for $q$.
Once $q$ is known, the relation $\num_x(z)=qm$ uniquely determines
$m$, and the conditions $y^q=z$ and $y^n=x^m$ in turn determine $y$:
since $q$ is prime to $n$, we can write $1=aq+bn$ with $a,b\in\bZ$
and then we find $y=z^ax^{bm}$. Thus $\varphi^{-1}(z)$ contains at
most $k+1$ elements $(y,q)$ of $P$.
\end{proof}

\begin{lemma}
 \label{inter:lemma2}
Let $E$ be a finite subset of $\bG$, let $k\ge 0$ be an integer, and
let $x\in\bG$ with $x\notin\Gtor$. Then, we have
\[
 |D_k(x,E) \cap \cO(E\setminus C_k(x,E))|
             \le (k+1) |C_{k+1}(x,E)|.
\]
\end{lemma}

\begin{proof}
It suffices to show that, for any $y\in E\setminus C_k(x,E)$ such
that $D_k(x,E)\cap \cO(y) \neq \emptyset$, we have $y\in
C_{k+1}(x,E)$ and $|D_k(x,E) \cap \cO(y)| \le k+1$.  Fix such a
choice of $y$ (assuming that there is one).  Since $D_k(x,E)\cap
\cO(y) \neq \emptyset$, there exist $p,q\in A$ and $z\in C_k(x,E)$
such that $y^q=z^p$. Moreover, since $z\in C_k(x,E)$, there also
exist $p_1,\dots,p_k,q_1,\dots,q_k\in A$ such that $z^{q_1\cdots
q_k} = x^{p_1\cdots p_k}$. Combining these two relations, we obtain
\begin{equation}
 \label{inter:lemma2:eq1}
 y^{qq_1\cdots q_k} = x^{pp_1\cdots p_k},
\end{equation}
which shows that $y\in C_{k+1}(x,E)$.   Put $n=\den_x(y)$ and
$m=\num_x(y)$.  By Lemma \ref{inter:lemma0}, the
equality \eqref{inter:lemma2:eq1} also implies that $n$ divides
$qq_1\dots q_k$ and that $m/n=(pp_1\cdots p_k)/(qq_1\cdots q_k)$. In
particular, the factorizations of $m$ and $n$ into prime numbers
have the same length: they involve the same number of elements of
$A$, counting multiplicities. If $j$ is this length, then the
equality $y^n=x^m$ means that $y\in C_j(x,E)$.  Since $y\notin
C_k(x,E)$, we must have $j>k$.  It follows that $j=k+1$ and
$n=qq_1\cdots q_k$. In particular, $q$ is one of the prime factors
of $n$.  Since $n=\den_x(y)$ has at most $k+1$ distinct prime
factors, we conclude that $|D_k(x,E) \cap \cO(y)| \le k+1$.
\end{proof}

\begin{proof}[Proof of Proposition \ref{inter:prop}]
We proceed by induction on $|E|$.  Fix a choice of $x\in E$.  We
claim that there exists an index $k$ with $0\le k\le \ell$ such that
the sets $C_k=C_k(x,E)$ and $D_k=D_k(x,E)$ satisfy
\begin{equation}
 \label{inter:prop:proof:eq1}
  |D_k\setminus\cO(E\setminus C_k)|
  \ge
  \frac{|A|-k}{2(k+1)} |C_k|.
\end{equation}
If we admit this statement, then, for such $k$, the sets $E_1 = C_k$
and $F_1 = D_k \setminus \cO(E\setminus C_k)$ fulfil the conditions
a), b) and c) of Proposition \ref{inter:prop} for $i=1$ and the
choice of $x_1=x$. Put $E' = E\setminus E_1$ and $F' = F\setminus
F_1$.  Then, we have $E=E_1\amalg E'$, $F=F_1\amalg F'$ and
$\cO(E')\subseteq F'$. If $E'=\emptyset$, this proves the
proposition with $r=1$ and $F_2=F'$. Otherwise, we may assume, by
induction, that the proposition applies to $E'$ and $F'$, and the
conclusion follows.

To prove the above claim, suppose on the contrary that
\eqref{inter:prop:proof:eq1} does not hold for any
$k=0,1,\dots,\ell$. Then, we have
\[
 |D_k|
  <
 |D_k\cap \cO(E\setminus C_k)| + \frac{|A|-k}{2(k+1)} |C_k|
 \qquad (0\le k\le \ell).
\]
Combining this with the lower bound for $|D_k|$ provided by Lemma
\ref{inter:lemma1} and the upper bound for $|D_k\cap \cO(E\setminus
C_k)|$ provided by Lemma \ref{inter:lemma2}, we obtain
\[
 \frac{|A|-k}{2(k+1)^2} |C_k| < |C_{k+1}|
 \qquad (0\le k\le \ell).
\]
Since $C_0=\{x\}$ has cardinality $1$, this leads to
$|C_{\ell+1}(x,E)| > (2^{\ell+1} (\ell+1)!)^{-1}
\binom{|A|}{\ell+1}$.  Then, by Lemma \ref{inter:lemma1}, we obtain
$|D_{\ell+1}(x,E)| > (2^{\ell+1} (\ell+1)!)^{-1}
\binom{|A|}{\ell+2}$.  This contradicts \eqref{inter:prop:eq1} since
$D_{\ell+1}(x,E)$ is a subset of $F$.
\end{proof}

\begin{remark}
It is easy to translate the proposition to the case of an abelian
group $\bG$ denoted additively.  Choose $\bG$ to be the additive
group of $\bQ$.  Let $s$ be a positive integer, let
$A=\{p_1,\dots,p_s\}$ be a set of $s$ distinct prime numbers, and
let $T$ be $A$-equivalence class of $1$ in $\bG=\bQ$.   Then,
Proposition \ref{inter:prop} applied to arbitrary subsets $E$ and
$F$ of $T$ with $\cO(E)\subseteq F$ translates into Proposition 6.2
of \cite{ixi}, upon identifying $\bZ^s$ with $T$ under the map which
sends a point $(i_1,\dots,i_s)\in\bZ^s$ to the rational number
$p_1^{i_1}\cdots p_s^{i_s}$.
\end{remark}

%
%
\section{Estimates for the gcd}
\label{sec:gcd}

We now apply the combinatorial result of the preceding section to
provide estimates for the degree and height of the greatest common
divisor of a family of polynomials of the form $P(T^a)$ where $P$ is
fixed and $a$ varies among a finite set of integers $A$.  The result
that we prove below implies Theorem \ref{intro:thm_gcd}.

\begin{theorem}
 \label{gcd:thm}
Let $K$ be a number field, let $M,n\in\bN^*$ with $M\ge 2$, let $A$
be a non-empty set consisting of prime numbers $p$ in the interval
$M/2\le p \le M$, let $P$ be a non-zero polynomial of $\KT$ of
degree at most $n$ with no root in $\Ctor\cup\{0\}$, and let
$Q\in\KT$ be a greatest common divisor of the polynomials $P(T^a)$
with $a\in A$. Suppose that there exists an integer $\ell$
satisfying
\[
 4\le 2\ell\le |A|
 \et
 n \le \frac{1}{2^{\ell+1}(\ell+1)!} \binom{|A|}{\ell+2}.
\]
Then, we have
\begin{equation}
 \label{gcd:thm:eq}
 \deg(Q)
 \le
 \frac{6\ell}{|A|} \deg(P)
 \et
 \log H(Q)
 \le
 \frac{c}{|A| M}\big(M \deg(P)+\log H(P)\big),
\end{equation}
with $c = \ell 2^{2\ell+6}$.
\end{theorem}

\begin{proof}
Suppose first that all roots of $P$ are simple.  Then, for each
$a\in A$, the roots of $P(T^a)$ are also simple (since $P(0)\neq
0$), and so the roots of $Q$ are simple.  Define $\bG$ to be the
multiplicative group $\Cmult$ of $\bC$, and let $E$ and $F$ denote
respectively the sets of roots of $Q$ and $P$. By hypothesis, we
have $F\subset \bG\setminus\Gtor$ and $|F| \le n$. Moreover, for any
$x\in E$ and any $a\in A$, $x$ is a root of $P(T^a)$ and so we have
$x^a\in F$. In the notation of \S\ref{sec:inter}, this means that $E
\subset \bG\setminus\Gtor$ and that $\cO(E)\subseteq F$. If
$E=\emptyset$, then $Q$ is a constant and \eqref{gcd:thm:eq} holds.
Otherwise, Proposition \ref{inter:prop} provides us with an integer
$r\ge 1$, a sequence of points $x_1,\dots,x_r$ of $E$, and
partitions $E = E_1 \amalg \cdots \amalg E_r$ and $F = F_1 \amalg
\cdots \amalg F_{r+1}$ satisfying, for $i=1,\dots,r$,
\begin{equation}
 \label{gcd:proofthm:eq1}
 E_i\subseteq C_\ell(x_i),
 \quad
 F_i\subseteq \cO(E_i)
 \et
 |F_i|
 \ge \frac{|A|-\ell}{2(\ell+1)}|E_i|
 \ge \frac{|A|}{6\ell} |E_i|.
\end{equation}
Summing term by term the last inequalities for $i=1,\dots,r$, we
obtain $|F|\ge |A|\, |E| / (6\ell)$ and so
\begin{equation}
 \label{gcd:proofthm:eq2}
 \deg(Q)
   = |E|
   \le \frac{6\ell}{|A|} |F|
   = \frac{6\ell}{|A|} \deg(P).
\end{equation}
For each $i=1,\dots,r$ and each point $x\in E_i$, we have $x\in
C_\ell(x_i)$ and so there exist $p_1, \dots, p_\ell, q_1, \dots,
q_\ell \in A$ such that $x_i^{p_1\cdots p_\ell} = x^{q_1\cdots
q_\ell}$. This gives $H(x_i)^{p_1\cdots p_\ell} = H(x)^{q_1\cdots
q_\ell}$, and thus
\begin{equation}
 \label{gcd:proofthm:eq3}
 2^{-\ell}\log H(x_i) \le \log H(x) \le 2^\ell\log H(x_i).
\end{equation}
Combining this with the standard estimates \eqref{prelim:ineq_roots}
for the height of a polynomial in terms of the height of its roots,
and using \eqref{gcd:proofthm:eq2} we deduce that
\begin{equation}
 \label{gcd:proofthm:eq4}
 \log H(Q)
  \le \deg(Q) + \sum_{x\in E} \log H(x)
  \le \frac{6\ell}{|A|} \deg(P)
     + \sum_{i=1}^r 2^\ell |E_i| \log H(x_i).
\end{equation}
On the other hand, for each $i=1,\dots,r$ and each $y\in F_i$, we
have $y\in\cO(E_i)$ and so there exist $a\in A$ and $x\in E_i$ such
that $y=x^a$.  Then, we get $H(y) = H(x)^a$, and by
\eqref{gcd:proofthm:eq3} we obtain
\[
 \log H(y)
 \ge \frac{M}{2}\log H(x)
 \ge \frac{M}{2^{\ell+1}}\log H(x_i).
\]
Combining this with \eqref{prelim:ineq_roots} and using
\eqref{gcd:proofthm:eq1}, we find
\[
 \log H(P) + \deg(P)
  \ge \sum_{y\in F} \log H(y)
  \ge \sum_{i=1}^r \frac{M}{2^{\ell+1}} |F_i| \log H(x_i)
  \ge \frac{|A| M}{6\ell 2^{\ell+1}}
      \sum_{i=1}^r |E_i| \log H(x_i).
\]
This provides an upper bound for $\sum_{i=1}^r |E_i| \log H(x_i)$
which after substitution into \eqref{gcd:proofthm:eq4} leads to
\[
 \log H(Q)
  \le \frac{c_1}{|A| M} \big( M \deg(P) + \log H(P) \big)
\]
with $c_1 = \ell 2^{2\ell+4} \ge 6\ell(1+2^{2\ell+1})$. This proves
the theorem with the constant $c$ replaced by $c_1$ when $P$ has
only simple roots.

In the general case, let $m$ denote the largest multiplicity of a
root of $P$.  For $i=1,\dots,m$, let $Z_i$ denote the set of roots
of $P$ having multiplicity at least $i$, and put $P_i=\prod_{x\in
Z_i}(T-x)$.  Since roots of $P$ which are conjugate over $K$ have
the same multiplicity, $P_1,\dots,P_m$ are polynomials of $\KT$.
Moreover, they have simple roots and $P$ is a constant multiple of
their product $P_1\cdots P_m$.  Put $Q_i=\gcd\{P_i(T^a)\,;\, a\in
A\}$ for $i=1,\dots,m$. We claim that $Q$ is a constant multiple of
$Q_1\cdots Q_m$.

To prove this claim, choose any root $x$ of $Q$.  We first observe
that, for each $a\in A$, the multiplicity of $x$ as a root of
$P(T^a)$ is the same as the multiplicity of $x^a$ as a root of $P$
(since $T^a-x^a$ has only simple roots).  Therefore the multiplicity
of $x$ as a root of $Q$ is the largest integer $i$ such that
$\cO(x)\subseteq Z_i$, or equivalently it is the largest integer $i$
such that $x$ is a root of each of the polynomials $Q_1,\dots,Q_i$.
This being true for each root $x$ of $Q$ shows that $Q$ divides
$Q_1\cdots Q_m$.  As the converse is clear, our claim follows.

Since $P_1,\dots,P_m$ all have degree at most $n$, the above
considerations show that the estimates \eqref{gcd:thm:eq} apply to
the pair $(Q_i,P_i)$ for each $i=1,\dots,m$, with $c$ replaced by
$c_1$. {From} this we deduce that
\[
 \deg(Q)
  = \sum_{i=1}^m \deg(Q_i)
  \le \sum_{i=1}^m \frac{6\ell}{|A|} \deg(P_i)
  = \frac{6\ell}{|A|} \deg(P)
\]
and
\[
 \begin{aligned}
 \log H(Q)
 &\le \deg(Q) + \sum_{i=1}^m \log H(Q_i) \\
 &\le \deg(Q) + \sum_{i=1}^m \frac{c_1}{|A| M}
     \big(M \deg(P_i)+\log H(P_i)\big) \\
 &\le \frac{c_1}{|A| M}
     \big((2M+1) \deg(P)+\log H(P)\big),
 \end{aligned}
\]
showing that \eqref{gcd:thm:eq} holds in general with $c = 4c_1$.
\end{proof}

%
%
\section{Proof of Theorem \ref{intro:thm1} for rank at least two}
\label{sec:mpoints}

Let the notation be as in Theorem \ref{intro:thm1}, and suppose that
$m\ge 2$.  For $\sigma=0$, the result follows from \cite[Prop.\
1]{LR}. So, we may assume that $\sigma>0$. Define positive constants
$\mu$ and $\epsilon$ by
\begin{equation}
 \label{mpoints:proof:eq1}
 \mu = \frac{m+1}{m+5}\,\sigma
 \et
 \epsilon
 = \frac{1}{8} \min\left\{ \sigma-\mu, \,\,
      \nu - 1 - \beta + \frac{3m-1}{m+5}\,\sigma + \tau \right\}.
\end{equation}
We proceed by contradiction, assuming on the contrary that for each
sufficiently large value of $n$ there exists a non-zero polynomial
$P\in\ZT$ with $\deg(P)\le n$ and $H(P)\le \exp(n^\beta)$ satisfying
\eqref{intro:thm1:eq2}. Upon dividing $P$ by its content, we may
assume that $P$ is primitive. Fix such an integer $n$ and a
corresponding polynomial $P$.  Each computation below assumes that
$n$ is larger than an appropriate constant depending only on
$\beta$, $\epsilon$, $\mu$, $\sigma$, $\tau$, $\nu$,
$\xi_1,\dots,\xi_m$, a condition that we write, in short, as $n\gg
1$. Define
\[
 t=\left[\frac{n^\tau+1}{2}\right], \
 d=\left[\frac{n}{t}\right], \
 \delta = \exp\left(-\frac{n^\nu}{6t}\right), \
 M=[n^\mu], \
 N=[n^\sigma], \
 \ 
 X=\exp(n^\beta),
\]
and factor $P$ as a product $P(T) = T^r \Phi(T)^t P_0(T)$ where $r$
is the largest non-negative integer such that $T^r$ divides $P(T)$,
and where $\Phi$ is the cyclotomic polynomial of $\ZT$ of largest
degree such that $\Phi^t$ divides $P$.  Since $\nu>1$, the main
condition \eqref{reduction:cor:eq} of Corollary \ref{reduction:cor}
is satisfied for $n\gg 1$ and so there exist relatively prime
positive integers $a_1, \dots, a_m, D$ with $D\le
(2mn^{2+\sigma})^m$ and a proper subspace $U$ of\/ $\bQ^m$ such that
we have $|\Phi(\xi_1^{i_1}\cdots\xi_m^{i_m})| \ge \delta$ for any
point $(i_1,\dots,i_m) \in \bZ^m \setminus U$ with
$\max\{|i_1|,\dots,|i_m|\} \le n^\sigma$ and
$\gcd(a_1i_1+\cdots+a_mi_m, D) =1$. If $\xi_1,\dots,\xi_m$ do not
all have absolute value one, we can further assume that
$a_1=\cdots=a_m=D=1$ by applying Proposition \ref{reduction:prop2}
instead.  Define
\[
 A = \{ a\in\cP \,;\, M/2 \le a\le M \ \text{and}\ a\ndiv D \}
\]
where $\cP$ denotes the set of all prime numbers, and define
\[
 E = \{ \xi_1^{i_1}\cdots\xi_m^{i_m} \,;\, (i_1,\dots,i_m)\in I\setminus
 (U\cup U')\},
\]
where $U'$ denotes the proper subspace of $\bQ^m$ generated by all
points $(i_1,\dots,i_m)\in\bZ^m$ for which $\xi_1^{i_1} \cdots
\xi_m^{i_m}$ is algebraic over $\bQ$, and where
\[
 I = \left\{ (i_1,\dots,i_m) \in \bZ^m \,;\,
     1 \le i_1,\dots,i_m \le n^{\sigma-\mu}
     \ \text{and}\
     \gcd(a_1i_1+\cdots+a_mi_m, D) =1 \right\}.
\]
Then, in the notation of Proposition \ref{first:prop}, we have
$\delta_\Phi \ge \delta$ and $\delta_P \le \exp(-n^\nu)$.  We claim
that for $n\gg 1$, we also have
\begin{equation}
 \label{mpoints:proof:eq2}
 n^{\mu-\epsilon} \le |A| \le n^\mu
 \et
 n^{m(\sigma-\mu)-\epsilon} \le |E| \le n^{m(\sigma-\mu)}.
\end{equation}
The upper bounds are clear and the lower bound for $|A|$ comes from
the prime number theorem.  The lower bound for $|E|$ follows from
\[
 |E| \ge |I| - |I\cap U| - |I\cap U'|
     \ge |I| - 2n^{(m-1)(\sigma-\mu)}
\]
together with the fact that, by Lemma \ref{counting:lemma:card_I}
(in the appendix), we have $|I| \ge 3n^{m(\sigma-\mu)-\epsilon}$ for
$n\gg 1$.  In particular, both sets $A$ and $E$ are not empty.  The
main conditions \eqref{first:prop:eq1} of Proposition
\ref{first:prop} also hold for $n\gg 1$ since we have
\begin{equation*}
 \tau + m(\sigma-\mu) < 1+\mu
 \et
 1+\mu < 1+\sigma \le \beta.
\end{equation*}
Therefore, according to this proposition, the polynomial
\begin{equation*}
 Q(T) = \gcd\{P_0^{[j]}(T^a) \,;\, a\in A,\, 0\le j<t\,\} \in \ZT
\end{equation*}
satisfies
\begin{equation*}
 \begin{aligned}
 \prod_{\xi\in E} \frac{|Q(\xi)|}{\cont(Q)}
 &\le X^{5Mn/t}\, \Delta_E^{-t}\,
     \left( \frac{\exp(-n^\nu)}{\delta^{3t}} \right)^{|E|} \\
 &\le \exp(15n^{1+\beta+\mu-\tau})\, \Delta_E^{-t}\,
     \exp(-n^\nu  |E|/2) \\
 &\le \exp(15n^{\nu+m(\sigma-\mu)-8\epsilon})\, \Delta_E^{-t}\,
     \exp(-n^\nu  |E|/2).
 \end{aligned}
\end{equation*}
Since $Q$ is primitive (being a divisor of $P(T^a)$ for any $a\in
A$), we conclude from \eqref{mpoints:proof:eq2} that for $n\gg 1$ we
have
\begin{equation*}
 \prod_{\xi\in E} |Q(\xi)|
 \le \exp\left( -n^\nu |E|/4 \right) \Delta_E^{-t}
 = \prod_{\xi\in E} \bigg( \exp\left( -n^\nu/4 \right)
   \prod_{\xi'\in E\setminus\{\xi\}} |\xi'-\xi|^{-t/2} \bigg).
\end{equation*}
Thus, there exists at least one point $\xi\in E$ such that
\begin{equation}
 \label{mpoints:proof:eq3}
 |Q(\xi)| \le \exp\left( -\frac{n^\nu}{8} \right)
 \quad\text{or}\quad
 \prod_{\xi'\in E\setminus\{\xi\}} |\xi'-\xi|
  \le \exp\left( -\frac{n^\nu}{4t} \right).
\end{equation}

Suppose for the moment that the first inequality in
\eqref{mpoints:proof:eq3} holds.  Denote by $P_1$ a divisor of $P$
in $\ZT$ of largest degree with no root in $\Ctor\cup\{0\}$, and
define
\[
 Q_1 = \gcd\big\{P_1(T^a) \,;\, a\in A\big\} \in \ZT.
\]
As $P_1$ divides $P$ in $\ZT$, we have $\deg(P_1)\le n$ and $\log
H(P_1) \le n + \log H(P) \le 2n^\beta$. Since $P_1$ has no root in
$\Ctor\cup\{0\}$, and since $|A|\ge n^{\mu-\epsilon} \ge n^{\mu/2}$
by \eqref{mpoints:proof:eq2}, Theorem \ref{gcd:thm} applies for
$n\gg 1$ with the choice of $\ell=[2/\mu]$, and it gives
\[
 \deg(Q_1) \le n^{1-\mu+2\epsilon}
 \et
 \log H(Q_1)  \le n^{\beta-2\mu+2\epsilon}.
\]
We claim moreover that $Q$ and $Q_1$ are related by
\begin{equation*}
 Q = \gcd\{Q_1^{[j]}(T)\,;\, 0\le j <t\}.
\end{equation*}
As $Q$ and $Q_1$ are primitive, this amounts to showing that their
orders of vanishing at any point $z\in \bC$ satisfy
\begin{equation}
 \label{mpoints:proof:eq5}
 \ord_z(Q) = \max\{0,\,\ord_z(Q_1)-t+1\}.
\end{equation}
To prove this, we first note that none of $P_0$ and $P_1$ vanishes
at $z=0$.  So the same is true for $Q$ and $Q_1$, and thus both
sides of \eqref{mpoints:proof:eq5} are $0$ when $z=0$. Assume from
now on that $z\in\Cmult$. Then we have
\begin{equation*}
 \ord_z(Q)
  = \min_{a\in A} \max\{0,\,\ord_{z^a}(P_0)-t+1\}
 \et
 \ord_z(Q_1)
  = \min_{a\in A} \ord_{z^a}(P_1).
\end{equation*}
If $z\in\Ctor$, we have $\ord_{z^a}(P_0)< t$ and $\ord_{z^a}(P_1) =
0$ for each $a\in A$, and then both sides of
\eqref{mpoints:proof:eq5} are again equal to $0$. Otherwise, we find
$\ord_{z^a}(P_0) = \ord_{z^a}(P_1)$ for each $a\in A$, and
\eqref{mpoints:proof:eq5} follows.

The above discussion shows that we may apply Lemma
\ref{lemma:linearization} to the pair of polynomials $Q$ and $Q_1$
with the function $\varphi\colon\ZT\to [0,\infty)$ given by
$\varphi(F)=|F(\xi)|$, and the choice of parameters
$d=n^{1-\mu+2\epsilon}$, $Y = \exp(n^{\beta-2\mu+2\epsilon})$ and
$\delta=\exp(-n^\nu/8)$. Assuming $n\gg 1$, this lemma ensures the
existence of a primary polynomial $S\in\ZT$ with
\[
 \deg(S) \le 4n^{1-\mu-\tau+2\epsilon},
 \quad
 \log H(S) \le 8n^{\beta-2\mu-\tau+2\epsilon}
 \et
 |S(\xi)| \le \exp(-n^{\nu-\tau-\epsilon}).
\]
We have $S(\xi)\neq 0$ since $S\neq 0$ and since $\xi$ is
transcendental over $\bQ$ (like all the elements of $E$). Write $\xi
= \xi_1^{i_1} \cdots \xi_m^{i_m}$ with exponents in the range $1\le
i_1,\dots,i_m\le n^{\sigma-\mu}$. Then $\tS(T_1,\dots,T_m) =
S(T_1^{i_1}\cdots T_m^{i_m})$ is a polynomial of
$\bZ[T_1,\dots,T_m]$ which for $n\gg 1$ satisfies
\begin{equation}
 \label{mpoints:proof:eq6}
 \begin{aligned}
  \deg(\tS) &\le n^{1+\sigma-2\mu-\tau+3\epsilon},\\
  \log H(\tS) &\le n^{\beta-2\mu-\tau+3\epsilon},\\
  0< |\tS(\xi_1,\dots,\xi_m)| &\le \exp(-n^{\nu-\tau-\epsilon}).
 \end{aligned}
\end{equation}

Suppose now that the second inequality holds in
\eqref{mpoints:proof:eq3}.  Then we have
\[
 \prod_{\xi'\in E\setminus\{\xi\}} |\xi'-\xi|
 = \tS(\xi_1,\dots,\xi_m)
\]
with $\tS \in \bZ[T_1,\dots,T_m]$ satisfying $\deg(\tS) \le m
n^{\sigma-\mu} |E|$, $\log H(\tS) \le |E|$ as well as the last
inequality of \eqref{mpoints:proof:eq6} when $n\gg 1$. Since
$(m+1)(\sigma-\mu) \le 1+\sigma-2\mu-\tau$, we deduce that $\tS$
also fulfills the first two inequalities of
\eqref{mpoints:proof:eq6} when $n\gg 1$.

Therefore the constraints \eqref{mpoints:proof:eq6} have a solution
$\tS \in \bZ[T_1,\dots,T_m]$ for each $n\gg 1$.  This contradicts
Lemma \ref{gelfond:curve} (Gel'fond's criterion) since we have
$\beta \ge 1+\sigma$ and since the choice of $\epsilon$ in
\eqref{mpoints:proof:eq1} implies
\[
 \nu-\tau-\epsilon
 \ge (1+\sigma-2\mu-\tau+3\epsilon)
   + (\beta-2\mu-\tau+3\epsilon) + \epsilon.
\]
The proof is complete.

%
%
\section{Avoiding cyclotomic factors in rank one}
\label{sec:rank_one}

The rest of this paper is devoted to the proof of Theorem
\ref{intro:thm1} in the case where $m=1$. In this section, we first
establish a measure of approximation of a complex number $\xi$ by
roots of unity, under conditions that are sensibly weaker than those
of Theorem \ref{intro:thm1}.  We then prove two corollaries which
finally allow us to push forward the conclusion of Proposition
\ref{first:prop}. The reader who simply wants a proof of Theorem
\ref{intro:thm1} in the case where $m=1$ and $|\xi_1|\neq 1$ can go
directly to the remark following those two corollaries and then
proceed to Proposition \ref{rank_one:prop2} at the end of the
section.

\begin{proposition}
 \label{rank_one:prop1}
Let $\xi\in\Cmult\setminus\Ctor$, and let $\beta, \sigma, \tau, \nu
\in \bR$ with
\[
 \sigma > 0, \quad
 \tau \ge 0, \quad
 \sigma+\tau \le 1 \le \beta
 \et
 \nu > 1+\beta-\sigma-\tau.
\]
Suppose that, for each sufficiently large positive integer $n$,
there exists a non-zero polynomial $P = P_n \in \ZT$ with $\deg(P)
\le n$ and $H(P) \le \exp(n^\beta)$ satisfying
\[
 \max\big\{ |P^{[j]}(\xi^i)| \,;\,
      1\le i\le n^\sigma,\ 0 \le j< n^\tau\,\big\}
 \le \exp(-n^\nu).
\]
Then, the ratio $\rho=(\nu-\tau)/(1-\tau)$ is a real number with
$\rho>1$ and, for each sufficiently large positive integer $D$ and
each root of unity $Z$ of order $D$, we have
\[
 |\xi-Z| \ge \exp\left( -\phi(D)^\rho \right).
\]
\end{proposition}

\begin{proof}
We have $\rho>1$ because $\nu>1>\tau$. Now, suppose on the contrary
that there exist roots of unity $Z$ of arbitrarily large order $D$
with $|\xi-Z| < \exp(-\phi(D)^\rho)$.  Fix such a pair $D$ and $Z$
and put $m = \phi(D)$. By taking $D$ large enough, we may assume
that the integer $n$ determined by the condition
\[
 2n^{1-\tau} < m \le 2(n+1)^{1-\tau}
\]
is arbitrarily large. In particular, we may assume that there exists
a corresponding polynomial $P=P_n\in\ZT$.  Furthermore, we may
assume that $P$ is primitive, so that $H(P)=\|P\|$.  Let $j\ge 0$ be
the smallest non-negative integer such that $P^{(j)}(Z)\neq 0$.
Since $Z$ has degree $m$ over $\bQ$, we have $jm \le \deg(P)\le n$
and so $j \le n/m < n^\tau/2$. Consider the polynomial $Q = P^{[j]}
\in \ZT$.  It has degree $\deg(Q)\le n$ and length $L(Q) \le (n+1)
2^n \|P\| \le \exp(3n^\beta)$. Since $Q(Z)$ is a non-zero algebraic
integer of $\bQ(Z)$, its norm from $\bQ(Z)$ to $\bQ$ is a non-zero
integer and so we have
\begin{equation}
 \label{rank_one:prop1:proof:eq1}
 1 \le \prod_{\substack{1\le i\le D\\ \gcd(i,D)=1}} |Q(Z^i)|
\end{equation}
Let $I$ denote the set of all integers $i$ coprime to $D$ with $1\le
i\le n^\sigma$.  Since $D\ge m > 2n^{1-\tau}\ge 2n^\sigma$, this is
a subset of the indexing set of the product in the right hand side
of \eqref{rank_one:prop1:proof:eq1}.  For each $i\in I$, we use the
Taylor expansion of $Q$ around $\xi^i$ to estimate $|Q(Z^i)|$. Fix
such an index $i$. This gives
\[
 |Q(Z^i)| \le \sum_{k=0}^\infty |Q^{[k]}(\xi^i)|\,|\xi^i-Z^i|^k.
\]
Since $m>2n^{1-\tau}$ and $\rho>1$, we have $|\xi-Z| < \exp(-m^\rho)
< \exp(-2 n^{\nu-\tau})$.  If $n$ is sufficiently large, we also
have $\exp(-2 n^{\nu-\tau}) \le n^{-\sigma}$, therefore $|\xi| \le 1
+ n^{-\sigma}$, and so $\max\{1,|\xi|\}^i \le e$ since $i\le
n^\sigma$. Combining these estimates, we obtain, for $n$
sufficiently large,
\[
 |\xi^i-Z^i|
  = |\xi-Z|\,
    \left|\sum_{\ell=0}^{i-1} \xi^\ell Z^{i-\ell-1} \right|
 \le \exp(-2 n^{\nu-\tau})\,n^\sigma e
 \le \exp(-n^{\nu-\tau}).
\]
In particular, we may assume that $|\xi^i-Z^i| \le 1/2$.  On the
other hand, since $j< n^\tau/2$, we have $j+k < n^\tau$ for any
integer $k$ with $0\le k \le n^\tau/2$, and for such an integer $k$
the hypothesis on $P$ leads to
\[
 |Q^{[k]}(\xi^i)|
 = \binom{j+k}{j}|P^{[j+k]}(\xi^i)|
 \le 2^n\exp(-n^\nu).
\]
For the remaining integers $k > n^\tau/2$, we use instead the crude
estimate
\[
 |Q^{[k]}(\xi^i)|
  \le \max\{1,\,|\xi^i|\}^n L(Q^{[k]})
  \le e^n 2^n L(Q)
  \le \exp(5n^\beta).
\]
So, putting all together, we find, for each $i\in I$,
\begin{align*}
 |Q(Z^i)|
  &\le 2^n \exp(-n^\nu) \sum_{k=0}^{[n^\tau/2]} |\xi^i-Z^i|^k
      + \exp(5n^\beta) \sum_{k=[n^\tau/2]+1}^\infty |\xi^i-Z^i|^k\\
  &\le 2^{n+1} \exp(-n^\nu) + 2\exp(5n^\beta)\, |\xi^i-Z^i|^{n^\tau/2}\\
  &\le 2^{n+1}\exp(-n^\nu) + 2\exp(5n^\beta - n^\nu/2)\\
  &\le \exp(-n^\nu/3),
\end{align*}
where the last step again assumes that $n$ is sufficiently large.
For all the other integers $i$, we use
\[
 |Q(Z^i)| \le L(Q) \le \exp(3n^\beta).
\]
Since the inequality \eqref{rank_one:prop1:proof:eq1} involves a
product of $m$ factors of the form $|Q(Z^i)|$, including those with
$i\in I$, we deduce that
\begin{equation}
 \label{rank_one:prop1:proof:eq2}
 1 \le \exp(3n^\beta)^m \exp(-n^\nu/3)^{|I|}.
\end{equation}
Define $\epsilon=(1/2)(\nu-1-\beta+\sigma+\tau) >0$.  We have $m\le
2(n+1)^{1-\tau} \le 4n^{1-\tau}$, and Lemma
\ref{counting:lemma:card_I} (or the prime number theorem) gives $|I|
\ge n^{\sigma-\epsilon}$ for $n$ sufficiently large, since $D =
\cO(m^2) = \cO(n^2)$. Substituting these estimates for $m$ and $|I|$
into \eqref{rank_one:prop1:proof:eq2} leads to a contradiction
because $\beta+1-\tau < \nu+\sigma-\epsilon$.
\end{proof}

\begin{corollary}
 \label{rank_one:cor1}
Under the notation and hypotheses of Proposition
\ref{rank_one:prop1}, there exists a positive integer $n_1$ with the
following property. For each pair of integers $n$ and $t$ with $n\ge
n_1$ and $t\ge n^\tau/3$, and for each cyclotomic polynomial
$\Phi\in\ZT$ whose $t$-th power $\Phi^t$ divides the polynomial
$P=P_n$, there exists a positive integer $D$ with $D\le 2n^3$ such
that
\begin{equation}
 \label{rank_one:cor1:eq}
 \min\{ |\Phi(\xi^i)| \,;\, 1\le i\le n^\sigma,\, \gcd(i,D)=1\}
 \ge
 \exp\left(-\frac{n^\nu}{6t}\right).
\end{equation}
\end{corollary}

\begin{proof}
Choose $\epsilon>0$ such that $\nu-\epsilon > 1 + \beta - \sigma -
\tau$.  Then the hypotheses of Proposition \ref{rank_one:prop1}
remain satisfied with the parameter $\nu$ replaced by
$\nu-\epsilon$, and so there exists a constant $c>0$ such that, for
any integer $D\ge 1$ and any root of unity $Z$ of order $D$, we have
\begin{equation}
 \label{rank_one:cor1:proof:eq1}
 |\xi-Z| \ge \exp(-c\phi(D)^{\trho})
 \quad
 \text{where}
 \quad
 \trho=\frac{\nu-\epsilon-\tau}{1-\tau}.
\end{equation}
Let $n$ be a positive integer for which the polynomial $P=P_n$ is
defined, let $t$ be an integer with $t\ge n^\tau/3$, and let $\Phi$
be a cyclotomic polynomial of $\ZT$ such that $\Phi^t$ divides $P$.
We may assume that $\Phi$ is non-constant, and so we have $t\le n$.
Then, for $n$ sufficiently large, all conditions of Proposition
\ref{cyclo:prop} are satisfied with $m=1$, $\xi_1=\xi$ and the
choice of parameters $d=[n/t]$, $\delta=\exp(-n^\nu/(6t))$ and
$N=[n^\sigma]$ (the condition \eqref{cyclo:prop:eq1} holds since
$\nu>1$).  So, there exist relatively prime positive integers $a_1$
and $D$ with $D\le 2(n/t)^2n^\sigma \le 2n^3$ such that either
\eqref{rank_one:cor1:eq} holds or there exists a root $Z$ of $\Phi$
which has order $D$ as a root of unity and satisfies
\begin{equation}
 \label{rank_one:cor1:proof:eq2}
 |\xi-Z^{a_1}|^G \le \exp\left( - \frac{n^\nu}{12t} \right)
\end{equation}
where $G$ denotes the multiplicity of $Z$ as a root of $\Phi$.
Suppose that the second eventuality holds.  We will see that, in
this case, the integer $n$ is bounded and this will complete the
proof. Since $Z$ and $Z^{a_1}$ are conjugate over $\bQ$ (they have
the same order $D$), we may assume without loss of generality that
$a_1=1$. Then, by comparing \eqref{rank_one:cor1:proof:eq1} and
\eqref{rank_one:cor1:proof:eq2}, we find
\begin{equation}
 \label{rank_one:cor1:proof:eq3}
 c G \phi(D)^{\trho} \ge \frac{n^\nu}{12 t}.
\end{equation}
However, since $Z$ has degree $\phi(D)$ over $\bQ$, we also have
$G\phi(D) \le \deg(\Phi) \le n/t$.  Combining this with
\eqref{rank_one:cor1:proof:eq3}, we get $12\,c\phi(D)^{\trho-1} \ge
n^{\nu-1}$.  Finally, since $\phi(D) \le n/t \le 3n^{1-\tau}$, this
gives $n \le (12\,c\,3^{\trho-1})^{1/\epsilon}$.
\end{proof}

\begin{corollary}
 \label{rank_one:cor2}
Let the notation and hypotheses be as in Proposition
\ref{rank_one:prop1}, and let $\mu\in \bR$ with $0<\mu\le 1-\tau$
and $2\mu+\tau < \nu$. Then, there exists a positive integer $n_2$
with the following property. For each integer $n\ge n_2$ and each
non-empty subset $I$ of $\{1,2,\dots,[n^\mu]\}$, the set $E=\{\xi^i
\,;\, i\in I\}$ satisfies
\[
 \Delta_E \ge \exp\left( - \frac{1}{4} n^{\nu-\tau} |E| \right).
\]
\end{corollary}

\begin{proof}
Again, choose $\epsilon>0$ such that $\nu-\epsilon > 1 + \beta -
\sigma - \tau$.  Arguing as in the proof of Corollary
\ref{rank_one:cor1}, we find that there is a constant $c>0$ such
that \eqref{rank_one:cor1:proof:eq1} holds for any integer $D\ge 1$
and any root of unity $Z$ of order $D$.  Let $n$ be a positive
integer and let $E=\{\xi^i\,;\, i\in I\}$ for some non-empty subset
$I$ of $\{1,2,\dots,[n^\mu]\}$.  Suppose that $\Delta_E < \exp(-
(1/4) n^{\nu-\tau} |E|)$.  We need to show that $n$ is bounded
(independently of the choice of $I$).  By definition, we have
$\Delta_E = \prod_{i<j} |\xi^i-\xi^j|$ where the product runs
through all pairs $(i,j)$ of elements of $I$ with $i<j$.  This means
that we can write $\Delta_E = |\xi|^r\,|\Phi(\xi)|$ for an integer
$r$ with $0\le r \le n^{2\mu} |E|$ and a cyclotomic polynomial
$\Phi$ of $\ZT$ of degree at most $n^{2\mu} |E|$.  Applying Lemma
\ref{cyclo:lemma3}, we deduce that some root $Z$ of $\Phi$ satisfies
\[
 |\xi-Z|^G
 \le
 \exp\big( n^{2\mu} |E| \log (c_1(n^{3\mu})^4)
           - (1/4) n^{\nu-\tau} |E| \big),
\]
where $c_1=2\max\{1,\,|\xi|^{-1}\}$ and where $G$ denotes the
multiplicity of $Z$ as a root of $\Phi$.  Since $\nu-\tau > 2\mu$,
we conclude that for $n$ large enough we have
\begin{equation}
 \label{rank_one:cor2:proof:eq1}
 |\xi-Z|^G \le \exp\big( -(1/8)n^{\nu-\tau} |E| \big).
\end{equation}
Now, let $D$ denote the order of $Z$ as a root of unity. Combining
this estimate with \eqref{rank_one:cor1:proof:eq1}, we obtain
\begin{equation}
 \label{rank_one:cor2:proof:eq2}
 8 c G \phi(D)^{\trho} \ge n^{\nu-\tau} |E|.
\end{equation}
On the other hand, because of the actual definition of $\Phi$, we
have $D\le n^\mu$, and $G$ is the number of pairs of elements
$(i,j)$ of $I$ with $i<j$ and $i\equiv j \mod D$. Thus, we also have
$G\le n^{\mu}|E|/D$.  Substituting this upper bound for $G$ into
\eqref{rank_one:cor2:proof:eq2} and using $\phi(D)\le D$, we obtain
$8 c D^{\trho-1} \ge n^{\nu-\mu-\tau}$.  Finally, since $D \le n^\mu
\le n^{1-\tau}$, this leads to $8 c n^{\nu-1-\epsilon} \ge
n^{\nu-\mu-\tau}$, thus $8c \ge n^{1+\epsilon-\mu-\tau} \ge
n^\epsilon$ and so $n \le (8c)^{1/\epsilon}$.
\end{proof}

\begin{remark}
If we assume that $|\xi|\neq 1$, then for each cyclotomic polynomial
$\Phi\in\ZT$ and each non-zero integer $i$, we find
\[
 |\Phi(\xi^i)|
 \ge \big| 1 - |\xi^i| \big|^{\deg(\Phi)}
 \ge c_1^{\deg(\Phi)},
\]
where $c_1=1-\min\{|\xi|,|\xi|^{-1}\}$.  Since $\nu > 1$, we deduce
that, in this case, Corollary \ref{rank_one:cor1} holds with $D=1$.
Moreover, for a set $E$ as in Corollary \ref{rank_one:cor2}, we have
$\Delta_E = |\xi|^r \Phi(\xi)$ where $r$ is an integer with $0\le r
\le n^{2\mu}|E|$ and $\Phi$ is a cyclotomic polynomial of $\ZT$ with
$\deg(\Phi) \le n^{2\mu} |E|$, and thus $\Delta_E \ge
\exp(c_2n^{2\mu}|E|)$ where $c_2=\log(c_1\min\{1,|\xi|\})$, which is
stronger than the conclusion of Corollary \ref{rank_one:cor2}.
\end{remark}

The main result of this section is the following.

\begin{proposition}
 \label{rank_one:prop2}
Let $\xi \in \Cmult\setminus \Ctor$, and let $\beta,\, \epsilon,\,
\mu,\, \sigma,\, \tau,\, \nu \in \bR$ with
\[
 \begin{gathered}
 0 < \mu < \sigma,
 \quad
 0\le \tau \le 1-\sigma,
 \quad
 \beta > \max\{1+\sigma-\mu,\, 2\mu+\tau\},
 \\
 0 < 2\epsilon < \min\{\mu,\, \sigma-\mu\}
 \et
 \nu > 1 + \beta + \sigma - 2\mu - \tau + 2\epsilon.
 \end{gathered}
\]
Suppose that for each sufficiently large positive integer $n$, there
exists a non-zero polynomial $P = P_n \in \bZ[T]$ with $\deg(P) \le
n$ and $H(P) \le \exp(n^\beta)$ satisfying
\[
 \max\big\{ |P^{[j]}(\xi^{i})| \,;\,
             1\le i\le n^\sigma,\, 0\le j <n^\tau \big\}
 < \exp(-n^\nu).
\]
Then, for each large enough index $n$, there exists an integer $D$
with $1\le D\le 2n^3$ satisfying the following property.  For any
set $I$ of cardinality $|I|\ge n^{\mu-\epsilon}$ consisting of
integers $i$ coprime to $D$ in the range $1\le i\le n^\mu$, there
exists a primary polynomial $S\in\ZT$ satisfying
\begin{equation}
 \label{rank_one:prop2:eq}
 \deg(S) \le n^{1-(\sigma-\mu)-\tau+3\epsilon}
 ,\
 \log H(S) \le n^{\beta-2(\sigma-\mu)-\tau+3\epsilon}
 ,\
 \prod_{i\in I} |S(\xi^i)| \le \exp(-n^{\nu+\mu-\tau-2\epsilon}).
\end{equation}
\end{proposition}

\begin{proof}
Fix a large integer $n$ and a corresponding polynomial $P$.  Without
loss of generality, we may assume that $P$ is primitive.  Put
$t=\left[(n^\tau+1)/2\right]$, and write $P(T)$ as a product $P(T) =
T^r \Phi(T)^t P_0(T)$, where $r$ is the largest positive integer
such that $T^r$ divides $P$ and $\Phi$ is the cyclotomic polynomial
of $\ZT$ of largest degree such that $\Phi^t$ divides $P$. Assuming
$n$ large enough, Corollary \ref{rank_one:cor1} shows that
\eqref{rank_one:cor1:eq} holds for some integer $D$ with $1\le D \le
2n^3$.  Let $I$ be a subset of $\{i\in\bZ\,;\, 1\le i\le n^\mu,\
\gcd(i,D)=1\}$ with cardinality $|I|\ge n^{\mu-\epsilon}$ (such a
subset exists if $n$ is large enough), and define $E = \{ \xi^i
\,;\, i\in I \}$.  Put also $M = [n^{\sigma-\mu}]$ and define $A$ to
be the set of all prime numbers $p$ not dividing $D$ with $M/2\le p
\le M$. Finally, set $X = \exp(n^\beta)$ so that $H(P)\le X$. Then,
in the notation of Proposition \ref{first:prop}, we have
\[
 c_E \le \exp(c_1 n^\mu),
 \quad
 \delta_\Phi \ge \exp\left(-\frac{n^\nu}{6t}\right)
 \et
 \delta_P \le \exp(-n^\nu),
\]
where $c_1=\log\max\{|\xi|,|\xi|^{-1}\}$.  Since $\xi \notin \Ctor
\cup \{0\}$, the sets $E$ and $I$ have the same cardinality.
Assuming $n$ large enough, we have
\[
 n^{\sigma-\mu-\epsilon} \le |A| \le n^{\sigma-\mu}
 \et
 n^{\mu-\epsilon} \le |E| = |I| \le n^{\mu}
\]
and the main condition \eqref{first:prop:eq1} of Proposition
\ref{first:prop} holds because $\tau + \mu < 1 + \sigma - \mu <
\beta$ and $2\mu +\tau < \beta$.  Combining this proposition with
Corollary \ref{rank_one:cor2}, it follows that the polynomial
\[
 Q(T) = \gcd\{ P_0^{[j]}(T^a) \,;\, a\in A,\, 0\le j <t\} \in \ZT
\]
satisfies
\begin{align*}
 \prod_{i\in I} |Q(\xi^i)|
 &\le \exp\left(\frac{5}{t} n^{1+\beta+\sigma-\mu}\right)
      \Delta_E^{-t}
      \exp\left( -\frac{n^\nu}{2} |E| \right)\\
 &\le \exp\left(15 n^{1+\beta+\sigma-\mu-\tau}\right)
      \exp\left( -\frac{n^\nu}{4} |E| \right)\\
 &\le \exp\left( - (1/8) n^{\nu+\mu-\epsilon} \right)
\end{align*}
provided that $n$ is large enough.

Denote by $P_1$ a divisor of $P$ in $\ZT$ of largest degree with no
root in $\Ctor\cup\{0\}$, and define
\[
 Q_1 = \gcd\big\{P_1(T^a) \,;\, a\in A\big\} \in \ZT.
\]
Applying Theorem \ref{gcd:thm} as in Section \ref{sec:mpoints}, upon
noting that $\beta\ge 1+\sigma-\mu$, we find that for $n$
sufficiently large we have
\[
 \deg(Q_1) \le n^{1-(\sigma-\mu)+2\epsilon}
 \et
 \log H(Q_1)  \le n^{\beta-2(\sigma-\mu)+2\epsilon}.
\]
As in Section \ref{sec:mpoints}, we also note that $Q =
\gcd\{Q_1^{[j]}(T)\,;\, 0\le j <t\}$.  This means that we may apply
Lemma \ref{lemma:linearization} to the pair of polynomials $Q$ and
$Q_1$ with the function $\varphi\colon\ZT\to [0,\infty)$ given by
$\varphi(F)=\prod_{i\in I}|F(\xi^i)|$, and the choice of parameters
\[
 d = n^{1-(\sigma-\mu)+2\epsilon},\quad
 Y = \exp(n^{\beta-2(\sigma-\mu)+2\epsilon})
 \et
 \delta=\exp(-(1/8) n^{\nu+\mu-\epsilon}).
\]
Assuming $n$ large enough, this lemma ensures the existence of a
primary polynomial $S\in\ZT$ with the required properties
\eqref{rank_one:prop2:eq}.
\end{proof}

%
%
\section{An estimate related to Zarankiewicz problem}
\label{sec:Zaran}

The following result is a strengthening of Proposition 9.1 of
\cite{ixi}.  As the latter, it has connection with a well-known
combinatorial problem of Zarankiewicz (see \cite[Chap.~12]{ES}).

\begin{proposition}
 \label{Zaran:prop}
Let $A$ and $B$ be finite non-empty sets, let $\kappa_1$ and
$\kappa_2$ be positive real numbers, and let $\varphi\colon A\times
B \to [0,\kappa_1]$ be any function on $A\times B$ with values in
the interval $[0,\kappa_1]$. Suppose that the inequality
\[
 \sum_{b\in B} \min\{\varphi(a_1,b),\varphi(a_2,b)\} \le
 \kappa_2
\]
holds for any pair of distinct elements $a_1$ and $a_2$ of $A$.
Then, we have
\[
 \sum_{a\in A}\sum_{b\in B} \varphi(a,b)
 \le
 \max\big\{ |A| \sqrt{2|B|\kappa_1\kappa_2},\
 2|B|\kappa_1 \big\}.
\]
\end{proposition}

The connection with the problem of Zarankiewicz is the following.
For positive integers $m$ and $n$, an $m\times n$ matrix $M$ with
coefficients in $\{0,1\}$ can be viewed as a function $\varphi\colon
A\times B \to \{0,1\}$ where $A=\{1,\dots,m\}$ and
$B=\{1,\dots,n\}$.  If, for some integer $n_1\ge 1$, the matrix $M$
contains no $2\times n_1$ sub-matrix consisting entirely of ones,
the hypotheses of the proposition are satisfied with $\kappa_1=1$
and $\kappa_2=n_1-1$ and consequently this matrix contains at most
$\max\big\{m\sqrt{2n(n_1-1)},\, 2n\big\}$ ones.

\begin{proof}
We first claim that for each $i=1,\dots,|A|$, we have
\begin{equation}
 \label{prop:Zaran:eq1}
 \sum_{a\in A}\sum_{b\in B} \varphi(a,b)
 \le
 \frac{|A||B|}{i}\, \kappa_1
  + \frac{(i-1)|A|}{2}\,\kappa_2.
\end{equation}
In the case where $i=|A|$, this follows from Proposition 9.1 of
\cite{ixi}. The proof of the general case proceeds by reduction to
this situation. Put $m=|A|$ and, for each $a\in A$, define
$\psi(a)=\sum_{b\in B} \varphi(a,b)$.  Choose also an ordering
$\{a_1,a_2,\dots,a_m\}$ of the elements of $A$ such that
$\psi(a_1)\ge \psi(a_2)\ge \cdots\ge \psi(a_m)$, and consider the
set $A'=\{a_1,\dots,a_i\}$.  Then, $A'$ and $B$ satisfy all the
hypotheses of the proposition for the restriction of $\varphi$ to
$A'\times B$, with the same values of $\kappa_1$ and $\kappa_2$.
Accordingly, by \cite[Prop. 9.1]{ixi}, we have
\[
 \sum_{a\in A'} \psi(a)
 =
 \sum_{a\in A'}\sum_{b\in B} \varphi(a,b)
 \le
 |B| \kappa_1 + \binom{i}{2} \kappa_2.
\]
On the other hand, since $\psi(a_j) \le (1/i)\sum_{a\in A'} \psi(a)$
for each $j = i+1, \dots, m$, we also find
\[
 \sum_{a\in A}\sum_{b\in B} \varphi(a,b)
 =
 \sum_{a\in A} \psi(a)
 \le
 \left( 1 +\frac{m-i}{i} \right) \sum_{a\in A'} \psi(a)
 =
 \frac{|A|}{i}\sum_{a\in A'} \psi(a).
\]
Our claim \eqref{prop:Zaran:eq1} follows by combining these two
estimates.

To conclude, put $\rho=2|B|\kappa_1/\kappa_2$.  If $\rho< |A|^2$, we
apply \eqref{prop:Zaran:eq1} with $i=[\sqrt{\rho}]+1$. This gives
\[
 \sum_{a\in A}\sum_{b\in B} \varphi(a,b)
 \le
 \frac{|A||B|}{\sqrt{\rho}}\,\kappa_1
  + \frac{|A|\sqrt{\rho}}{2}\,\kappa_2
 =
 |A|\sqrt{\rho}\,\kappa_2
 =
 |A|\sqrt{2|B|\kappa_1\kappa_2}.
\]
If $\rho \ge |A|^2$, the same inequality with $i=|A|$ leads to
\[
 \sum_{a\in A}\sum_{b\in B} \varphi(a,b)
 \le
 |B|\kappa_1 + \frac{|A|^2}{2} \kappa_2
 \le
 |B|\kappa_1 + \frac{\rho}{2} \kappa_2
 =
 2|B|\kappa_1.
\]
The proof is complete.
\end{proof}

%
%
\section{Products of values of polynomials at powers of $\xi$}
\label{sec:product}

In this section, we use Proposition \ref{Zaran:prop} to prove a
transcendence criterion for a complex number $\xi$, based on
products of values of polynomials at powers of $\xi$. Then we
combine this criterion with Proposition \ref{rank_one:prop2} to
complete the proof of Theorem \ref{intro:thm1} in the case $m=1$.

\begin{theorem}
 \label{product:thm}
Let $\xi\in\bC$ be a transcendental number, and let $\alpha, \beta,
\mu, \omega \in \bR$ with
\begin{equation}
 \label{product:thm:eq1}
 \alpha \ge \mu >0,
 \quad
 \beta \ge \alpha + \mu
 \et
 \omega > \alpha+\beta +(3/2)\mu.
\end{equation}
For infinitely many positive integers $n$, there exists no primary
polynomial $Q\in\ZT$ without root in $\Ctor\cup\{0\}$ satisfying
\begin{equation}
 \label{product:thm:eq2}
 \deg(Q) \le n^\alpha,
 \quad
 H(Q) \le \exp(n^\beta)
 \et
 \prod_{a\in A}\,\prod_{b\in B} |Q(\xi^{ab})| \le \exp(-n^\omega),
\end{equation}
for some non-empty subsets $A$ and $B$ of
$\{1,2,\dots,[n^{\mu/2}]\}$.
\end{theorem}

\begin{proof}
We proceed by contradiction, assuming on the contrary that such a
triple $(Q,A,B)$ exists for each sufficiently large $n$. Fix an
appropriate integer $n$, and define $E=\{\xi^b\,;\, b\in B\}$ for a
corresponding choice of $(Q,A,B)$.  Note that $Q$ is primitive being
primary and non-constant, thus $H(Q)=\|Q\|$.  We consider two cases
according to the size of $\Delta_E$ (see \S \ref{sec:prelim} for the
definition of this quantity).

\smallskip
\noindent\textbf{Case 1:} $\Delta_E^{-1} \le
\exp((1/4)n^{\omega-\mu})$.

\smallskip
We claim that, if $n$ is sufficiently large, there exists $(a,b)\in
A\times B$ such that $|Q(\xi^{ab})| \le
\exp((-(1/2)n^{\omega-\mu/2})$. To prove this, we first note that,
for each $(a,b)\in A\times B$, we have
\[
 |Q(\xi^{ab})|
    \le \|Q\| \exp(c_1n^{\alpha+\mu})
    \le \exp((c_1+1)n^\beta),
\]
where $c_1= \log (1+|\xi|)$, so that we can write
\[
 |Q(\xi^{ab})| = \exp((c_1+1)n^\beta-\varphi(a,b))
\]
for some real number $\varphi(a,b)\ge 0$. This defines a function
$\varphi\colon A\times B \to [0,\infty)$ which, by the last
condition of \eqref{product:thm:eq2}, satisfies
\begin{equation}
 \label{product:thm:proof:eq1}
 \sum_{a\in A} \sum_{b\in B} \varphi(a,b)
 \ge
 n^\omega.
\end{equation}

We also note that, for distinct elements $a_1$ and $a_2$ of $A$, the
polynomials $Q(T^{a_1})$ and $Q(T^{a_2})$ are relatively prime in
$\ZT$.  This is because, they are primitive polynomials of $\ZT$
and, if $z$ is a common root of them, then $z^{a_1}$ and $z^{a_2}$
are roots of $Q(T)$.  However, since $Q(T)$ is a primary polynomial
of $\ZT$, its roots are conjugate over $\bQ$.  So, there exists an
automorphism $\sigma$ of the splitting field of $Q(T)$ over $\bQ$
such that $\sigma(z^{a_1}) = z^{a_2}$.  Then, upon denoting by $m$
the order of $\sigma$, we find $z^{a_1^m} = \sigma^m\big( z^{a_1^m}
\big) = z^{a_2^m}$.  Since $a_1^m\neq a_2^m$, this implies that
$z\in\Ctor\cup\{0\}$, contrary to the assumption that $Q(T)$ has no
root in that set.  Thus, the gcd of $Q(T^{a_1})$ and $Q(T^{a_2})$ in
$\ZT$ is $1$.

We apply Proposition \ref{prelim:prop:resultant} to the above
situation with $t=1$, $r=2$ and $P_i(T)=Q(T^{a_i})$ for $i=1,2$.
Since both polynomials $P_1$ and $P_2$ have degree $\le
n^{\alpha+\mu/2}$ and height $\le \exp(n^\beta)$, and since we
assume that $\Delta_E^{-1} \le \exp((1/4)n^{\omega-\mu})$, it gives
\[
 \begin{aligned}
 1 \le \exp&\big( 10n^{2\alpha+\mu} + c_2n^{\alpha+3\mu/2} +
            (1/4)n^{\omega-\mu} + 2n^{\beta+\alpha+\mu/2} \big)\\
    & \times \prod_{b\in B} \max\big\{
    \exp((c_1+1)n^\beta-\varphi(a_1,b)),\,
    \exp((c_1+1)n^\beta-\varphi(a_2,b)) \big\}
 \end{aligned}
\]
where $c_2=4\log(2+|\xi|)$.  By \eqref{product:thm:eq1}, the
exponent $\omega-\mu$ exceeds all the other exponents of powers of
$n$ in the first factor on the right.  So, if $n$ is sufficiently
large, we deduce that
\[
 \sum_{b\in B} \min\{\varphi(a_1,b)),\,\varphi(a_2,b))\}
 \le
 \frac{1}{2}n^{\omega-\mu}.
\]

This means that Proposition \ref{Zaran:prop} applies to the function
$\varphi$ with $\kappa_1$ equals to the largest value of $\varphi$
on $A\times B$, and with $\kappa_2=(1/2)n^{\omega-\mu}$. Because of
\eqref{product:thm:proof:eq1}, this implies that
\[
 n^\omega \le \max\{ n^{\mu/2}\sqrt{n^{\omega-\mu/2}\kappa_1},\,
 2n^{\mu/2}\kappa_1\},
\]
and so $\kappa_1 \ge (1/2)n^{\omega-\mu/2}$.  Thus, there exists
$(a,b)\in A\times B$ such that
\[
 |Q(\xi^{ab})| \le \exp((c_1+1)n^\beta-(1/2)n^{\omega-\mu/2}).
\]
If $n$ is sufficiently large, this means that $|Q(\xi^{ab})| \le
\exp(-(1/4)n^{\omega-\mu/2})$, thereby proving our claim.  For such
a choice of $(a,b)$, the polynomial $S(T) = Q(T^{ab}) \in \ZT$
satisfies
\begin{equation}
 \label{product:thm:proof:eq2}
 \deg(S) \le n^{\alpha+\mu},
 \quad
 H(S) \le \exp(n^\beta)
 \et
 0< |S(\xi)| \le \exp(-(1/4)n^{\omega-\mu/2}).
\end{equation}

\smallskip
\noindent\textbf{Case 2:} $\Delta_E^{-1}>
\exp((1/4)n^{\omega-\mu})$.

\smallskip
In this situation, we define
\[
 S(T) = \prod_{\substack{b,b'\in B \\ b<b'}} |T^{b'}-T^b|^u
 \quad \text{where} \quad
 u = [n^{\mu/2}] +1.
\]
This polynomial of $\ZT$ fulfills the inequalities
\eqref{product:thm:proof:eq2} because
\[
 \deg(S)
   \le \binom{|B|}{2} n^{\mu/2}u
   \le n^{2\mu} \le n^{\alpha+\mu},
  \quad
 \log H(S)
   \le \binom{|B|}{2} u \log(2)
   \le n^{3\mu/2} \le n^{\beta},
\]
and, by definition of $\Delta_E$, we have $0< |S(\xi)| = \Delta_E^u
\le \exp(-(1/4)n^{\omega-\mu/2})$.

\smallskip
Thus, in both cases, the conditions \eqref{product:thm:proof:eq2}
have a solution $S(T)\in \ZT$ for $n$ sufficiently large. By
Gel'fond's criterion (Lemma \ref{gelfond:curve}), this is impossible
because $\beta\ge \alpha+\mu$ and $\omega-\mu/2 >
(\alpha+\mu)+\beta$.
\end{proof}

\begin{proof}[Proof of Theorem \ref{intro:thm1} in the case $m=1$]
Suppose that the hypotheses of Theorem \ref{intro:thm1} are
satisfied for $m=1$. For $\sigma=0$, the result follows from
\cite[Prop.~1]{LR}.  We may therefore assume that $\sigma>0$.
Arguing by contradiction, we also assume that, for each sufficiently
large positive integer $n$, there exists a non-zero polynomial
$P\in\ZT$ with $\deg(P)\le n$ and $H(P)\le \exp(n^\beta)$ satisfying
\eqref{intro:thm1:eq2}. Put $\xi=\xi_1$ and $\mu=(8/11)\sigma$.
Then, the conditions of Proposition \ref{rank_one:prop2} are
fulfilled for any choice of $\epsilon>0$ small enough as a function
of $\beta,\sigma,\tau,\nu$. For each sufficiently large $n$ and for
the corresponding integer $D$ with $1\le D\le 2n^3$ provided by
Proposition \ref{rank_one:prop2}, consider the set of all prime
numbers $p$ not dividing $D$ in the interval $1< p \le n^{\mu/2}$,
and partition this set into two disjoint subsets $A$ and $B$ of
cardinality at least $n^{(\mu-\epsilon)/2}$. Then, the set
$I=\{ab\,;\, a\in A,\, b\in B\}$ has cardinality $|I| = |A|\/|B| \ge
n^{\mu-\epsilon}$ and consists of integers coprime to $D$ from the
interval $[1,n^\mu]$. So, Proposition 9.4 provides us with a primary
polynomial $S\in\ZT$ satisfying the conditions
\eqref{rank_one:prop2:eq}.  This contradicts Theorem
\ref{product:thm} if, from the start, we choose $\epsilon$ small
enough so that the conditions \eqref{product:thm:eq1} hold with
$\alpha = 1 - (\sigma-\mu) - \tau + 3\epsilon$, $\omega = \nu +\mu -
\tau- 2\epsilon$, and $\beta$ replaced by $\beta - 2(\sigma-\mu) -
\tau + 3\epsilon$.
\end{proof}

%
%

\appendix
\section{Counting lemmas}
\label{sec:counting}

The purpose of this appendix is to provide an estimate that is
needed in \S\ref{sec:mpoints} in the course of the proof of the main
Theorem \ref{intro:thm1} for the case $m\ge 2$.  It concerns the
cardinality of certain subsets of $\bZ^m$ which arise from an
application of Corollary \ref{reduction:cor}.  I believe that this
has appeared elsewhere but as I have been unable to find a suitable
reference, I include the details of proof for the convenience of the
reader.  It starts with a preliminary lemma.

\begin{lemma}
 \label{counting:lemma:congr}
Let $m, d, N \in \bN^*$, and let $a_1,\dots,a_m, b \in \bZ$ with
$\gcd(a_1,\dots,a_m,d)=1$. Then, the set
\[
  I = \{ (i_1,\dots,i_m)\in\bZ^m \,;\, 1\le i_1,\dots,i_m\le N \
  \text{and }\ a_1i_1+\cdots+a_mi_m \equiv b \mod d \}
\]
has cardinality $|I| = N^m/d + E$ with an error $E$ satisfying $|E|
\le (3N)^{m-1}$.
\end{lemma}

The crucial point here is that the error term depends only on $m$
and $N$.

\begin{proof}
If $m=1$, the set $I$ is the intersection of $\{1,2,\dots,N\}$ with
an arithmetic progression with difference $d$.  Therefore, its
cardinality is either $[N/d]$ or $[N/d]+1$, and so we have $\big|
|I| - N/d \big| \le 1$.  Suppose now that $m\ge 2$.  Write
$d_1=\gcd(a_1,d)$, $d'=d/d_1$ and $a'=a_1/d_1$, and define
\[
  I' = \{ (i_2,\dots,i_m)\in\bZ^{m-1} \,;\, 1\le i_2,\dots,i_m\le N \
  \text{and }\ a_2i_2+\cdots+a_mi_m \equiv b \mod d_1 \}.
\]
For any point $(i_1,\dots,i_m) \in \bZ^m$ we have $(i_1,\dots,i_m)
\in I$ if and only if $(i_2,\dots,i_m)\in I'$ and
\begin{equation}
 \label{counting:lemma:congr:eq1}
 a' i_1 \equiv (b-a_2i_2-\cdots-a_mi_m)/d_1 \mod d'
 \quad\text{with} \quad
 1\le i_1\le N.
\end{equation}
By the preceding considerations (case $m=1$), for fixed
$(i_2,\dots,i_m) \in I'$ the set of solutions $i_1$ of
\eqref{counting:lemma:congr:eq1} has cardinality $N/d' +
E(i_2,\dots,i_m)$ with $|E(i_2,\dots,i_m)| \le 1$.  From this we
deduce that
\[
 |I|
  = \sum_{(i_2,\dots,i_m)\in I'}
    \Big( \frac{N}{d'} + E(i_2,\dots,i_m) \Big)
  = \frac{N}{d'} |I'| + E'
 \quad\text{with} \quad
 |E'| \le |I'|.
\]
Since $\gcd(a_2,\dots,a_m,d_1)=1$, we can also assume by induction
that $|I'| = N^{m-1}/d_1 + E''$ with $|E''| \le (3N)^{m-2}$.
Combining these estimates gives $|I| = N^m/d + E$ with
\[
 |E| \le |I'| + N|E''| \le N^{m-1} + (N+1)(3N)^{m-2} \le
 (3N)^{m-1}.
\]
\end{proof}

The main estimate is the following.

\begin{lemma}
 \label{counting:lemma:relprime}
Let $m,D,N\in\bN^*$, and let $a_1,\dots,a_m\in\bZ$ with $\gcd(a_1,
\dots, a_m, D) = 1$. Then,
\[
  I = \{ (i_1,\dots,i_m)\in\bZ^m \,;\, 1\le i_1,\dots,i_m\le N \
  \text{and}\ \gcd(a_1i_1+\cdots+a_mi_m, D)=1 \}
\]
has cardinality
\[
 |I| = N^m\prod_{p|D}\Big(1-\frac{1}{p}\Big) + E
 \quad \text{with} \quad
 |E| \le 2^{\omega(D)}(3N)^{m-1},
\]
where the product runs over all prime factors $p$ of $D$ and where
$\omega(D)$ stands for the number of distinct prime factors of $D$.
\end{lemma}

\begin{proof}
For each positive divisor $d$ of $D$, define
\[
  I_d = \{ (i_1,\dots,i_m)\in\bZ^m \,;\, 1\le i_1,\dots,i_m\le N \
  \text{and}\ d | a_1i_1+\cdots+a_mi_m \}.
\]
Since $\gcd(a_1, \dots, a_m, d) = 1$, the preceding lemma gives
$|I_d| = N^m/d + E_d$ with $|E_d| \le (3N)^{m-1}$.  In terms of the
Moebius function $\mu$, the inclusion-exclusion principle gives
\[
 |I| = \sum_{d|D} \mu(d) |I_d|.
\]
The conclusion then follows from the fact that $\sum_{d|D} \mu(d)
d^{-1} = \prod_{p|D} (1-p^{-1})$ and that $D$ admits exactly
$2^{\omega(D)}$ square-free positive divisors.
\end{proof}

In the present paper, we use the estimate of the above lemma in the
following form.

\begin{lemma}
 \label{counting:lemma:card_I}
Let the notation be as in Lemma \ref{counting:lemma:relprime}, and
let $\epsilon$ and $\kappa$ be positive real numbers such that $D\le
N^\kappa$.  If $N$ is sufficiently large in terms of $\epsilon$,
$\kappa$ and $m$, then the set $I$ has cardinality at least
$N^{m-\epsilon}$.
\end{lemma}

\begin{proof}
By Lemma \ref{counting:lemma:relprime}, we have $|I| \ge N^m
2^{-\omega(D)} - 2^{\omega(D)} (3N)^{m-1}$. Since $2^{\omega(D)} =
\cO(D^{\delta})$ for any fixed $\delta>0$ (see \cite[Thm 315]{HW}),
we also find that $2^{\omega(D)} \le 2^{\omega(D)} 3^{m-1} \le
(1/2)N^\epsilon$ if $N$ is sufficiently large in terms of
$\epsilon$, $\kappa$ and $m$. As we may assume that $\epsilon \le
1/2$, this gives $|I| \ge 2N^{m-\epsilon}-(1/2)N^{m-1+\epsilon} \ge
N^{m-\epsilon}$.
\end{proof}

%
%


\begin{thebibliography}{99}



\bibitem{Br}
 W.~D.~Brownawell,
 Sequences of Diophantine approximations,
 {\it J.\ Number Theory} {\bf 6} (1974), 11-21.

\bibitem{Ch}
  C.~Chevalley, {\it Introduction to the theory of algebraic functions of
  one variable}, American Math.\ Soc., 1951.

\bibitem{ES}
  P.~Erd\"os and J.~Spencer
  {\it Probabilistic methods in combinatorics},
   Probability and Mathematical Statistics {\bf 17},
   Academic Press, 1974.

\bibitem{HW}
  G.~H.~Hardy and E.~M.~Wright,
  {\it An introduction to the theory of numbers}, fifth edition,
  Clarendon Press, Oxford, 1985.

\bibitem{LR}
  M.~Laurent and D.~Roy,
  Criteria of algebraic independence with multiplicities and
  interpolation determinants,
  {\it Trans.\ Amer.\ Math.\ Soc.\ }{\bf 351} (1999), 1845--1870.


\bibitem{R2001}
   D.~Roy,
   An arithmetic criterion for the values of the exponential
   function,
   {\it Acta Arith.\ }{\bf 97} (2001), 183--194.

\bibitem{ixi}
  D.~Roy,
  Small value estimates for the additive group,
  {\it Intern.\ J.\ Number Theory}, {\bf 6} (2010), 919--956;
  arXiv:0708.2307v1 [math.NT].

\bibitem{Sc}
  W.~M.~Schmidt,
  {\it Diophantine approximation}, Lecture Note in Math.,
  vol.~785, Sprin\-ger-Verlag, 1980.

\bibitem{Wa1}
  M.~Waldschmidt,
  Solution du huiti\`eme probl\`eme de Schneider,
  {\it J.\ Number Theory} {\bf 5} (1973), 191--202.

\end{thebibliography}
\end{document}